\documentclass[12pt]{amsart}
\usepackage{amsmath,latexsym,amssymb}
\usepackage{enumerate}
\usepackage{amsthm}
\usepackage{epsfig,graphicx}
\usepackage{tikz}
\usepackage{mathrsfs}

\newcommand\NN{\mathbb{N}}
\newcommand\RR{\mathbb{R}}

\newcommand\CC{\mathscr{C}}

\newcommand\II{\mathcal{I}}

\newcommand\TT{\mathcal{T}}

\newcommand\eps{{\varepsilon}}

\newtheorem{theorem}{Theorem}[section]

\newtheorem{corollary}[theorem]{Corollary}
\newtheorem{example}[theorem]{Example}
\newtheorem{lemma}[theorem]{Lemma}
\newtheorem{remark}[theorem]{Remark}
\newtheorem*{remarks*}{Remarks}
\newtheorem{proposition}[theorem]{Proposition}

\numberwithin{equation}{section}

\begin{document}

\title[The H\"older spectrum of a self-affine function]{The pointwise H\"older spectrum of general self-affine functions on an interval}

\author{Pieter Allaart}
\address{Mathematics Department, University of North Texas, 1155 Union Cir \#311430, Denton, TX 76203-5017, U.S.A.}
\email{allaart@unt.edu}


\begin{abstract}
This paper gives the pointwise H\"older (or multifractal) spectrum of continuous functions on the interval $[0,1]$ whose graph is the attractor of an iterated function system consisting of $r\geq 2$ affine maps on $\mathbb{R}^2$. These functions satisfy a functional equation of the form $\phi(a_k x+b_k)=c_k x+d_k\phi(x)+e_k$, for $k=1,2,\dots,r$ and $x\in[0,1]$. They include the Takagi function, the Riesz-Nagy singular functions, Okamoto's functions, and many other well-known examples. It is shown that the multifractal spectrum of $\phi$ is given by the multifractal formalism when $|d_k|\geq |a_k|$ for at least one $k$, but the multifractal formalism may fail otherwise, depending on the relationship between the shear parameters $c_k$ and the other parameters. In the special case when $a_k>0$ for every $k$, an exact expression is derived for the pointwise H\"older exponent at any point. These results extend recent work by the author [{\em Adv. Math.} {\bf 328} (2018), 1--39] and S. Dubuc [{\em Expo. Math.} {\bf 36} (2018), 119--142].
\end{abstract}




\maketitle

\section{Introduction} \label{sec:intro}

If $f:[0,1]\to \RR$ is a continuous function, then for $\xi\in(0,1)$ and $\alpha>0$ we write $f\in \CC^\alpha(\xi)$ if there exist a constant $C$ and a polynomial $P$ of degree less than $\alpha$ such that
\begin{equation}
|f(x)-P(x)|\leq C|x-\xi|^\alpha \qquad\mbox{for all $x\in[0,1]$}.
\label{eq:pointwise-Holder-property}
\end{equation}
The {\em pointwise H\"older exponent} of $f$ at $\xi$ is the number
\[
\alpha_f(\xi):=\sup\{\alpha\geq 0: f\in \CC^\alpha(\xi)\}.
\]
H\"older exponents were used by Mandelbrot \cite{Mandelbrot} and Frisch and Parisi \cite{Frisch} to study intermittent turbulence, and were investigated later from a mathematical point of view by Jaffard \cite{Jaffard1,Jaffard2} and others, mostly using wavelet methods. In this paper we use a more direct approach to determine the pointwise H\"older spectrum of the general self-affine function $\phi:[0,1]\to\RR$, defined as follows.
Let $\{(x_k,y_k): k=0,1,\dots,r\}$ be a polygonal line such that $0=x_0<x_1<x_2<\dots<x_r=1$, and let $T_1,\dots,T_r$ be affine transformations of $\RR^2$ given by $T_k(x,y)=(a_k x+b_k,c_k x+d_k y+e_k)$, $k=1,2,\dots,r$, where $0<|a_k|<1$, $|d_k|<1$, and for each $k$, 
\begin{equation}
T_k(\{(x_0,y_0),(x_r,y_r)\})=\{(x_{k-1},y_{k-1}),(x_k,y_k)\}.
\label{eq:connectivity}
\end{equation}
Note that \eqref{eq:connectivity} implies $\sum_{k=1}^r |a_k|=1$. There is then a unique continuous function $\phi$ satisfying the functional equation
\begin{equation}
\phi(a_k x+b_k)=c_k x+d_k\phi_k(x)+e_k, \qquad x\in[0,1], \quad k=1,2,\dots,r.
\label{eq:functional-equation}
\end{equation}
(See \cite[Theorem 1]{Dubuc}.) Many famous ``pathological functions" are obtained in this way; see the examples below.
The objective of this paper is to determine the pointwise H\"older spectrum of $\phi$; that is, the function
\[
D(\alpha):=\dim_H E_\phi(\alpha), \qquad \mbox{where} \qquad E_\phi(\alpha):=\{\xi\in[0,1]: \alpha_\phi(\xi)=\alpha\}, \quad \alpha>0.
\]
The function $D(\alpha)$ has previously been determined for some special cases; for instance, Ben Slimane \cite{BenSlimane} gave a complete solution for $r=2$ and $a_1=a_2=1/2$. More recently, the author \cite{Allaart} considered functions of the form \eqref{eq:functional-equation} without shear coefficients (i.e. $c_k=0$ for each $k$) and showed that $D(\alpha)$ is given by the multifractal formalism (defined below) provided that $\alpha_\phi(\xi)$ is replaced with
\[
\tilde{\alpha}_\phi(\xi):=\sup\left\{\alpha\geq 0: \limsup_{x\to\xi}\frac{|\phi(x)-\phi(\xi)|}{|x-\xi|^\alpha}=0\right\}.
\]
This quantity is easier to analyze but says little about the regularity of $\phi$ near $\xi$ if $\phi$ has a nonvanishing finite derivative at $\xi$. The present article improves the main result of \cite{Allaart} in two important ways: (i) It extends the result to self-affine functions with shears (i.e. $c_k$ not all zero); and (ii) it provides the spectrum of the actual pointwise H\"older exponent $\alpha_\phi$ rather than $\tilde{\alpha}_\phi$. As will be seen below, the presence of shears sometimes causes the multifractal formalism to fail, as was already noted in \cite{BenSlimane}.

This paper also complements recent work of Dubuc \cite{Dubuc}, who classifies the differentiability properties of all functions $\phi$ of the form \eqref{eq:functional-equation} and shows that 
\[
\alpha_\phi(\xi)=\frac{\sum_{k=1}^r |a_k|\log|d_k|}{\sum_{k=1}^r |a_k|\log|a_k|} \qquad \mbox{almost everywhere in $(0,1)$},
\]
unless $\phi$ happens to be a polynomial. (This generalizes an earlier result of Bedford \cite{Bedford}). The results and proof techniques from \cite{Allaart} and \cite{Dubuc} play an integral role in the proofs of the main result below, but several new ideas had to be introduced to deal with the cases when the multifractal formalism fails.

\subsection{Examples}

Before stating the main result, we present some well-known examples of self-affine functions.
Observe first that for any given polygonal line $\{(x_k,y_k): k=0,1,\dots,r\}$ there are many self-affine functions of the form \eqref{eq:functional-equation}. For instance, when $a_k>0$ for every $k$, \eqref{eq:connectivity} is equivalent to the system
\begin{align}
\begin{split}
a_k&=x_k-x_{k-1},\\
b_k&=x_{k-1},\\
d_k(y_r-y_0)+c_k&=y_k-y_{k-1},\\
d_k y_0+e_k&=y_{k-1},
\end{split}
\label{eq:parameter-relations}
\end{align}
so there is one degree of freedom for each $k$. One may, for example, freely choose all the vertical (signed) contraction ratios $d_1,\dots,d_r$, and then the other parameters are fixed.

\begin{example}[Generalized Takagi functions] \label{ex:Takagi}
{\rm
For a parameter $w>0$, let
\[
\phi_w(x):=\sum_{n=0}^\infty \frac{\varphi(2^n x)}{2^{wn}},
\]
where $\varphi(x)$ denotes the distance from $x$ to the nearest integer. Then $\phi_w$ satisfies \eqref{eq:functional-equation} with $r=2$, $a_1=a_2=c_1=-c_2=1/2$, and $d_1=d_2=2^{-w}$. The case $w=1$ gives the classical Takagi function (see \cite{Takagi}), shown in the leftmost panel of Figure \ref{fig:self-affine-examples}. When $0<w<1$, $\phi_w$ is nowhere differentiable and its graph has Hausdorff dimension greater than $1$ \cite{Ledrappier}; and when $w>1$, $\phi_w$ is differentiable everywhere except at the dyadic rationals in $[0,1]$. Note that $w=2$ gives the parabola $\phi_2(x)=2x(1-x)$, which is uninteresting from the point of view of multifractal analysis. In most of the results of this paper, we will have to explicitly rule out the possibility that $\phi$ is a polynomial.
}
\end{example}

\begin{example}
{\rm
The Riesz-Nagy singular function \cite{Riesz-Nagy} with parameter $a\in(0,1)$ (middle graph in Figure \ref{fig:self-affine-examples}) satisfies \eqref{eq:functional-equation} with $a_1=a_2=1/2$, $c_1=c_2=0$, and $d_1=1-d_2=a$. The Riesz-Nagy function has also been attributed to Salem \cite{Salem} and Hellinger \cite{Hellinger}, and is sometimes called Lebesgue's singular function (e.g. \cite{Kawamura}).
}
\end{example}

\begin{example}
{\rm
Okamoto's function \cite{Okamoto} with parameter $a\in(0,1)$ (rightmost graph in Figure \ref{fig:self-affine-examples}) satisfies \eqref{eq:functional-equation} with $a_1=a_2=a_3=1/3$, $c_1=c_2=c_3=0$, $d_1=d_3=a$, and $d_2=1-2a$. For $a=1/2$ we get the Cantor function; the case $a=2/3$ was studied by Katsuura \cite{Katsuura}.
}
\end{example}

\begin{figure}
\begin{center}
\epsfig{file=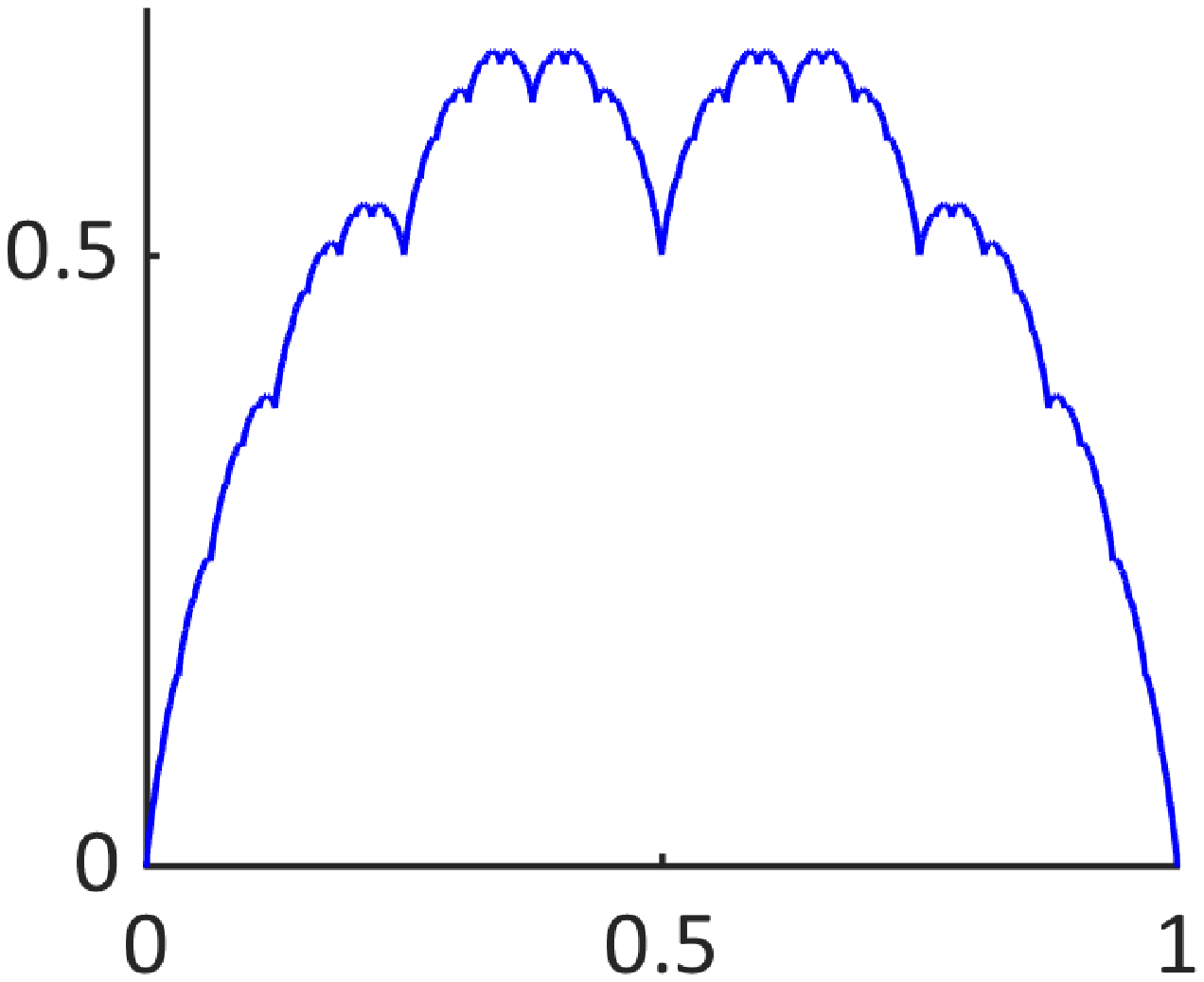, height=.2\textheight, width=.3\textwidth} \quad
\epsfig{file=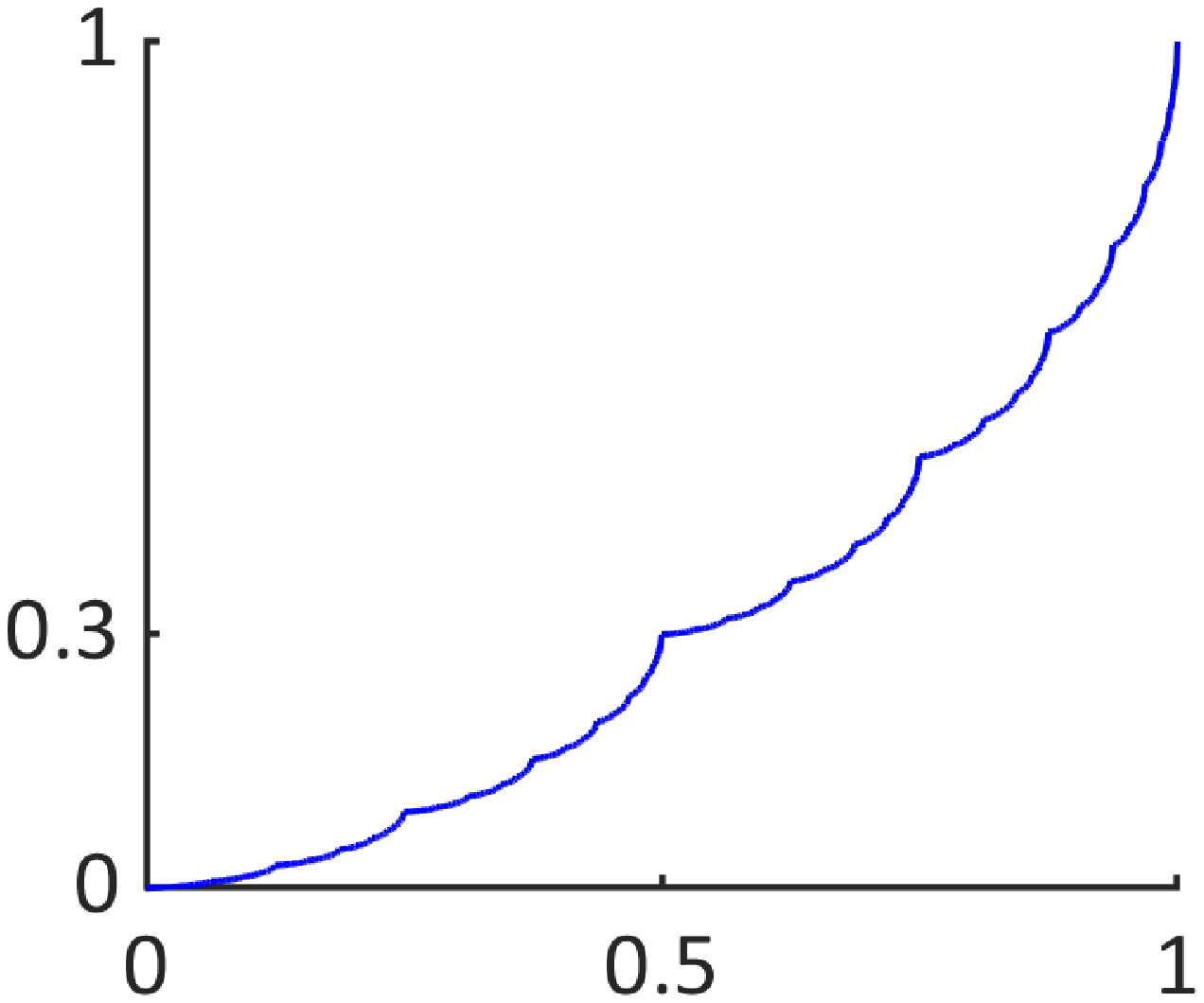, height=.2\textheight, width=.3\textwidth} \quad
\epsfig{file=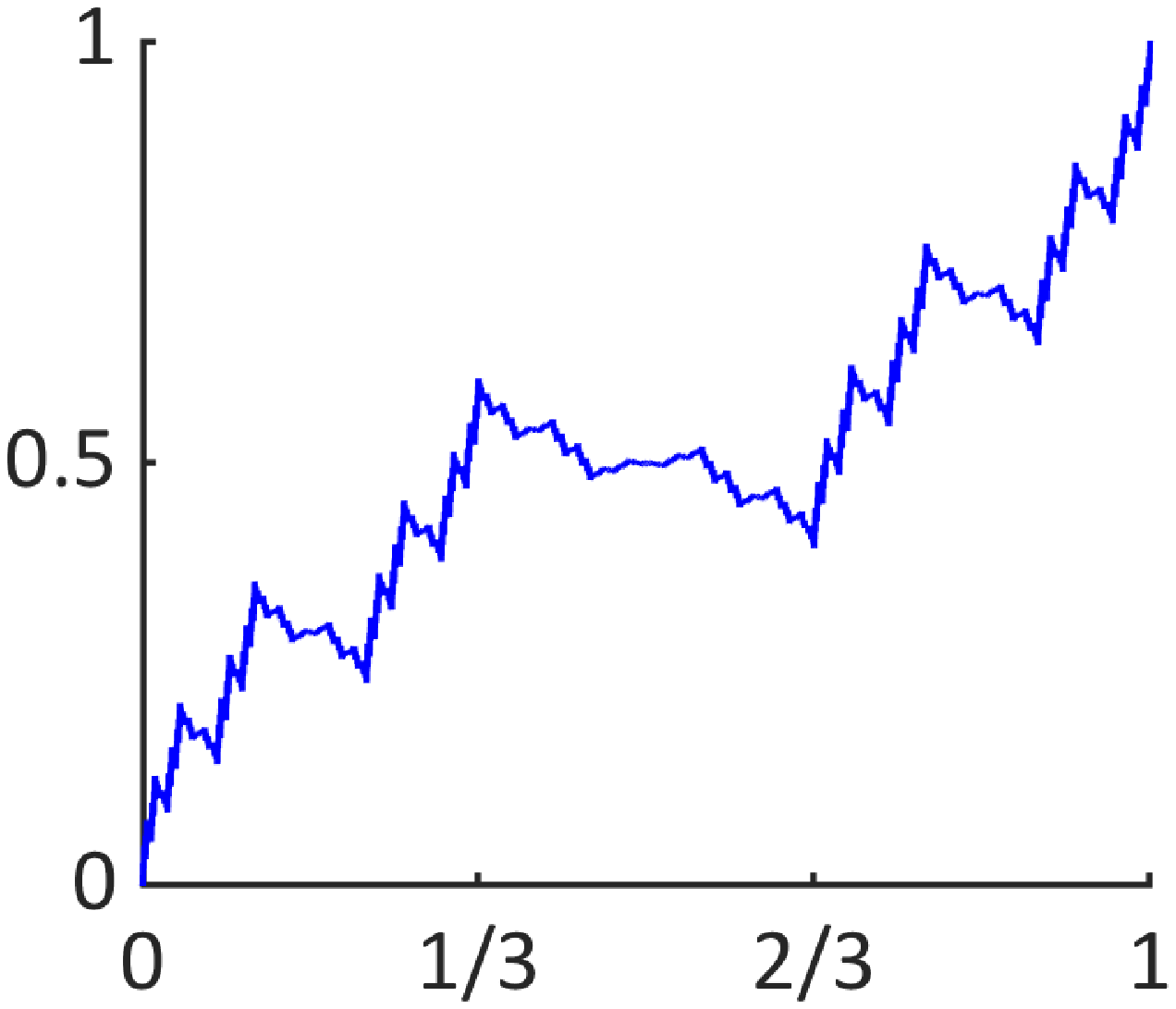, height=.2\textheight, width=.3\textwidth}
\caption{Three well known self-affine functions: The Takagi function (left), the Riesz-Nagy singular function with $a=0.3$ (center) and Okamoto's function with $a=0.6$ (right)}
\label{fig:self-affine-examples}
\end{center}
\end{figure}

Many more examples, as well as a useful Matlab program for drawing graphs of self-affine functions (which the author gratefully used to produce  Figure \ref{fig:self-affine-examples}), can be found in Dubuc's expository paper \cite{Dubuc}.


\subsection{The main result}

Define the index sets
\begin{equation*}
\mathcal{I}:=\{1,2,\dots,r\}, \qquad \mathcal{I}_0:=\{k\in\mathcal{I}:d_k=0\}, \qquad \mathcal{I}_+:=\{k\in\mathcal{I}:d_k\neq 0\}.
\end{equation*}
To avoid degenerate cases, we assume $\#\II_+\geq 2$. Set
\begin{equation*}
\rho_k:=\frac{\log|d_k|}{\log|a_k|}, \quad k\in\II_+,
\end{equation*}
and let
\begin{equation*}
\alpha_{\min}:=\min_{k\in\II_+}\rho_k, \qquad \alpha_{\max}:=\max_{k\in\II_+}\rho_k.
\end{equation*}
Furthermore, let $s_{\min}$, $s_{\max}$ and $\hat{s}$ be the nonnegative numbers satisfying
\begin{equation*}
\sum_{k: \rho_k=\alpha_{\min}}|a_k|^{s_{\min}}=1, \qquad \sum_{k: \rho_k=\alpha_{\max}}|a_k|^{s_{\max}}=1, \qquad \sum_{k\in\II_+}|a_k|^{\hat{s}}=1.
\end{equation*}
Note that $\hat{s}>0$ since $\#\II_+\geq 2$. On the other hand, $s_{\min}$ and $s_{\max}$ are typically zero unless there is a tie for the minimum (resp. maximum) of $\log|d_k|/\log|a_k|$. Put
\begin{equation*}
\hat{\alpha}:=\frac{\sum_{k\in\II_+}|a_k|^{\hat{s}}\log|d_k|}{\sum_{k\in\II_+}|a_k|^{\hat{s}}\log|a_k|}.
\end{equation*}
Note that $\alpha_{\min}\leq \hat{\alpha}\leq \alpha_{\max}$.

For each $q\in\RR$, let $\beta(q)$ be the unique real number such that
\begin{equation}
\sum_{k\in\II_+}|d_k|^q |a_k|^{\beta(q)}=1.
\label{eq:scaling-equation}
\end{equation}
It is well known from multifractal theory (e.g. \cite[Chapter 17]{Falconer}) that the function $\beta(q)$ is strictly decreasing and convex, and its Legendre transform
\begin{equation}
\beta^*(\alpha):=\inf_{q\in\RR}\{\alpha q+\beta(q)\}, \qquad \alpha>0
\label{eq:legendre-transforms}
\end{equation}
is strictly concave on the interval $[\alpha_{\min},\alpha_{\max}]$, and takes the value $-\infty$ outside this interval.
We say the {\em multifractal formalism} (MFF) holds if $D(\alpha)=\beta^*(\alpha)$ for all $\alpha$. In what follows, we shall assume that $a_k>0$ for each $k$. It is straightforward, but cumbersome, to state similar theorems for cases where some of the $a_k$ are negative.

Define the index set
\begin{align}
\begin{split}
\Lambda&:=\left\{k\in\{1,2,\dots,r-1\}: c_{k+1}a_k-c_k a_{k+1}+\frac{c_1 d_{k+1}a_k}{a_1-d_1}-\frac{c_r d_k a_{k+1}}{a_r-d_r}\neq 0\right.\\
&\left.\phantom{\frac{c_1 d_{k+1}a_k}{a_1-d_1}}\qquad\mbox{and}\ \max\{|d_k|,|d_{k+1}|\}>0\right\}.
\end{split}
\label{eq:Lambda}
\end{align}
Here we interpret $0/0$ as zero, and we declare that $k\in\Lambda$ if $a_1=d_1$ and $c_1 d_{k+1}\neq 0$, or similarly, if $a_r=d_r$ and $c_r d_k\neq 0$.

A special case of our main theorem below (with $r=2$ and $a_1=a_2=1/2$) was proved by Ben Slimane \cite{BenSlimane}.

\begin{theorem} \label{thm:main}
Assume $a_k>0$ for $k=1,2,\dots,r$.
\begin{enumerate}[(a)]\setlength{\itemsep}{2mm}
\item Suppose $|d_k|\geq a_k$ for at least one $k$ or $\Lambda=\emptyset$. Then \vspace{2mm}
\begin{enumerate}[(i)]\setlength{\itemsep}{2mm}
\item $E_\phi(\alpha)=\emptyset$ when $\alpha\not\in[\alpha_{\min},\alpha_{\max}]\cup\{\infty\}$;
\item $E_\phi(\infty)$ is empty if $\II_0=\emptyset$, and has Lebesgue measure one otherwise;
\item $D(\alpha)=\beta^*(\alpha)$ for all $\alpha\in(\alpha_{\min},\alpha_{\max})$;
\item $D(\alpha_{\min})=s_{\min}$, and $D(\alpha_{\max})=s_{\max}$;
\item The maximum value of $D(\alpha)$ is attained at $\hat{\alpha}$, and
$D(\hat{\alpha})=\hat{s}$. Moreover, if $\II_0=\emptyset$, then $E_\phi(\hat{\alpha})$ has Lebesgue measure one.
\end{enumerate}
\item Suppose $|d_k|<a_k$ for every $k$ and $\Lambda\neq\emptyset$. Let $\sigma>0$ be such that
\[
\sum_{k\in\II_+}\left(\frac{|d_k|}{a_k}\right)^\sigma=1,
\]
and put
$p_k^*:=(|d_k|/a_k)^\sigma$ for $k\in\II_+$.
Let
\[
\alpha_0:=\frac{\sum_{k\in\II_+} p_k^*\log|d_k|}{\sum_{k\in\II_+} p_k^*\log a_k}.
\]
\begin{enumerate}[(i)]\setlength{\itemsep}{2mm}
\item $E_\phi(\alpha)=\emptyset$ when $\alpha\not\in[1,\alpha_{\max}]\cup\{\infty\}$;
\item $E_\phi(\infty)$ is empty if $\II_0=\emptyset$, and has Lebesgue measure one otherwise;
\item For $\alpha\in[1,\alpha_{\max}]$ we have
\begin{equation}
D(\alpha)=\begin{cases}
\sigma(\alpha-1) & \mbox{if $1\leq \alpha\leq\alpha_0$},\\
\beta^*(\alpha) & \mbox{if $\alpha_0\leq\alpha\leq\alpha_{\max}$}.
\end{cases}
\label{eq:special-spectrum}
\end{equation}
\item The function $D(\alpha)$ is differentiable at $\alpha_0$.
\item $D(\alpha_{\max})=s_{\max}$.
\item $\alpha_0\leq\hat{\alpha}$, and statement (v) of part (a) holds.
\end{enumerate}
\end{enumerate}
\end{theorem}

The multifractal spectrum of Theorem \ref{thm:main}\,({\em b}) is illustrated in Figure \ref{fig:excpetional-spectrum}. That $D(\alpha)$ does not follow the MFF for $\alpha<\alpha_0$ can be roughly explained as follows: The presence of shears in the maps $T_k$ causes a drop in H\"older exponent for points which are exceptionally well approximated by the images of $0$ and $1$ under the composite maps $S_{k_1}\circ \dots \circ S_{k_n}$, for $k_1,\dots,k_n\in\{1,2,\dots,r\}$, where $S_k(x):=a_k x+b_k$. This phenomenon occurs only for H\"older exponents greater than 1, but in the setting of Theorem \ref{thm:main}\,({\em b}) it affects a large enough set of points to increase the value of $D(\alpha)$ at the lower end of the spectrum, compared to what the MFF would predict (the dashed line in Figure \ref{fig:excpetional-spectrum}).

\begin{figure}
\begin{center}
\begin{tikzpicture}[scale=3]
\draw [->] (0,0)--(2.8,0) node[anchor=west] {$\alpha$};
\draw [->] (0,0)--(0,1.2);
\draw[domain=1.82:2.499,smooth,variable=\x] plot({\x},{(-(\x-1.5)*ln(\x-1.5)-(1-\x+1.5)*ln(1-\x+1.5))/ln(2)});
\draw[domain=1.501:1.82,smooth,dashed,variable=\x] plot({\x},{(-(\x-1.5)*ln(\x-1.5)-(1-\x+1.5)*ln(1-\x+1.5))/ln(2)});
\draw (0 cm,1pt)--(0 cm,-1pt) node[anchor=north] {$0$};
\draw (1 cm,1pt)--(1 cm,-1pt) node[anchor=north] {$1$};
\draw (1.5 cm,1pt)--(1.5 cm,-1pt) node[anchor=north] {$\alpha_{\min}$};
\draw (2.5 cm,1pt)--(2.5 cm,-1pt) node[anchor=north] {$\alpha_{\max}$};
\draw (1.82 cm,1pt)--(1.82 cm,-1pt) node[anchor=north] {$\alpha_0$};
\draw[fill] (1.82,0.904) circle [radius=0.02];
\draw (2 cm,1pt)--(2 cm,-1pt); 
\node [below] at (2,-0.02) {$\hat{\alpha}$};
\draw (1pt,1cm)--(-1pt,1cm) node[anchor=east] {$1$};
\draw (1,0)--(1.82,0.904);
\node [above] at (1.3,0.6) {$\sigma(\alpha-1)$};
\node [above] at (2.55,0.6) {$\beta^*(\alpha)$};
\end{tikzpicture}
\end{center}
\caption{Typical graph of $D(\alpha)$ when $|d_k|<a_k$ for all $k$ and $\Lambda\neq\emptyset$}
\label{fig:excpetional-spectrum}
\end{figure}
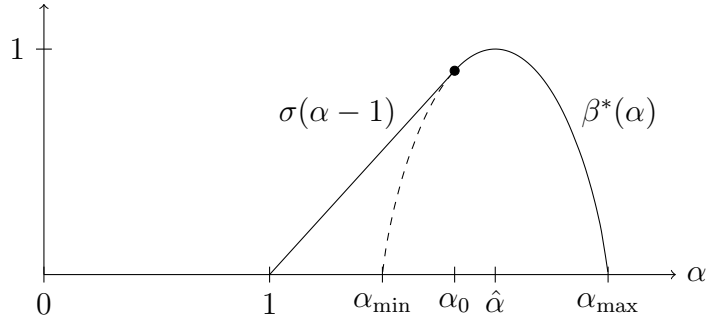

\begin{remarks*}
{\rm
(i) Note that the $c_k$'s are used to determine whether $\Lambda=\emptyset$, but play no further role in the determination of the multifractal spectrum of $\phi$.

(ii) When $r=2$, the set $\Lambda$ contains at most the number $1$. It is easy to check that in this case (assuming $d_1<a_1$ and $d_2<a_2$),
\begin{equation}
\Lambda=\emptyset \qquad \Longleftrightarrow \qquad \frac{c_1}{a_1-d_1}=\frac{c_2}{a_2-d_2} \quad\mbox{or}\quad \frac{d_1}{a_1}+\frac{d_2}{a_2}=1.
\label{eq:Lambda-for-two-maps}
\end{equation}

(iii) Theorem \ref{thm:main}\,({\em b}) includes the following special case: If $\log|d_k|/\log a_k$ is constant (for all $k\in\II_+$), then $\alpha_{\max}=\alpha_{\min}=\hat{\alpha}$ and $p_k^*=a_k^{\hat{s}}$ for each $k\in\II_+$, so that $\alpha_0=\hat{\alpha}$ as well. In this case, the graph of $D(\alpha)$ is just the straight line segment connecting the points $(1,0)$ and $(\hat{\alpha},\hat{s})$.

(iv) When $a_k<0$ for some $k$, the set $\Lambda$ has to be modified. Other than that, however, the statement of the theorem remains essentially the same.
}
\end{remarks*}

\subsection{Related literature}

There is a vast body of literature on H\"older spectra of continuous functions; we mention only those papers here which are most closely related to the present article. Jaffard \cite{Jaffard2} determines the multifractal spectrum of ``self-similar" functions $\RR^d\to\RR$, but imposes a certain smoothness condition that rules out our functions. However, Jaffard and Mandelbrot \cite{JafMan} adapt Jaffard's method to compute the multifractal spectrum of P\'olya's space-filling curve, which is of the form \eqref{eq:functional-equation} except that it maps $[0,1]$ into $\RR^2$. Ben Slimane \cite{BenSlimane} gives a complete description of the H\"older spectrum of $\phi$ in the case $r=2$ and $a_1=a_2=1/2$. (This includes the Riesz-Nagy function and the generalized Takagi functions from Example \ref{ex:Takagi}.) Self-affine functions mapping $[0,1]$ into $\RR^d$ for any $d\geq 1$ were studied by Allaart \cite{Allaart} and B\'ar\'any et al.~\cite{BKK}. These functions satisfy a functional equation $f(a_k x+b_k)=\Psi_k(f(x))$ for $k=1,2,\dots,r$ and $x\in[0,1]$, where $\Psi_k:\RR^d\to\RR^d$ are similarities in \cite{Allaart}, and general affine maps in \cite{BKK}. (Not surprisingly, the price for greater generality in \cite{BKK} is that the results are less explicit, and severe restrictions must be imposed to obtain the full pointwise H\"older spectrum.) Recently, Jaerisch and Sumi \cite{JS} have extended some of the results of \cite{Allaart} to distribution functions of general Gibbs measures.

\subsection{Organization of the paper}

The rest of this article is organized as follows. In Section \ref{sec:prelim} we define a kind of generalized entropy function and state two useful duality principles. Section \ref{sec:examples} gives two concrete examples illustrating the main theorem. Section \ref{sec:exact-Holder-exponent}, the longest and most technical section of the paper, develops an exact expression for the {\em right pointwise H\"older exponent} of $\phi$ at any point $\xi$. A crucial role in its proof is played by the method of divided differences that was introduced in this context by Dubuc \cite{Dubuc}. This expression is then used in the next two sections, which prove the lower and upper bounds in Theorem \ref{thm:main}, respectively. The differentiability of $D(\alpha)$ at $\alpha_0$ is proved in Section \ref{sec:smooth-connection}. Finally, in Section \ref{sec:higher-dim}, we show how the techniques of this paper can be used to strengthen the main result of \cite{Allaart}.

\section{Generalized entropy and duality principles} \label{sec:prelim}

For the remainder of this paper we shall assume without further mention that $a_k>0$ for every $k$.

Before illustrating the main theorem, we present $\beta^*(\alpha)$ and the function in \eqref{eq:special-spectrum} as constrained maxima of certain entropy-like functions over a simplex. These dual representations will be important for the proofs later and can also be convenient for concrete computations. We denote by $\Delta_r$ the standard simplex in $\RR^r$:
\begin{equation*}
\Delta_r:=\left\{\mathbf{p}=(p_1,\dots,p_r)\in \RR^r: p_k\geq 0\ \mbox{for each $k$ and}\ \sum_{k=1}^r p_k=1\right\},
\end{equation*}
and set
\begin{equation*}
\Delta_r^0:=\{\mathbf{p}=(p_1,\dots,p_r)\in\Delta_r: p_k=0\ \mbox{for $k\in\mathcal{I}_0$}\}.
\end{equation*}
Define the function
\begin{equation}
H(\mathbf{p}):=\frac{\sum_{k=1}^r p_k\log p_k}{\sum_{k=1}^r p_k\log |a_k|}, \qquad \mathbf{p}=(p_1,\dots,p_r)\in\Delta_r,
\label{eq:h-definition}
\end{equation}
where as usual, we set $0\log 0\equiv 0$. A proof of the following useful duality principle can be found in \cite{Allaart}.

\begin{proposition} \label{prop:duality}
For each $\alpha\in[\alpha_{\min},\alpha_{\max}]$, we have
\begin{equation*}
\beta^*(\alpha)
=\max\bigg\{H(\mathbf{p}): \mathbf{p}\in\Delta_r^0,\ \sum_{k\in\mathcal{I_+}}p_k(\log|d_k|-\alpha\log|a_k|)=0\bigg\}.
\label{eq:duality}
\end{equation*}
\end{proposition}

Proposition \ref{prop:duality} represents $\beta^*(\alpha)$ as the maximum value of $H$ over the intersection of a simplex with a hyperplane. The characterization is especially useful when $r=2$, in which case the intersection consists of a single point, which is the solution of two linear equations in the two unkwowns $p_1$ and $p_2$. 

In order to prove Theorem \ref{thm:main}\,({\em b}), we need a second, somewhat more involved duality principle. Recall that $\sigma$ is the unique number satisfying $\sum_{k\in\II_+}(|d_k|/a_k)^\sigma=1$.

\begin{lemma} \label{lem:constrained-max}
Assume $|d_k|<a_k$ for each $k$. The maximum value of
\[
G(p_1,\dots,p_r):=\frac{\sum_{k\in\II_+}p_k\log p_k}{\sum_{k\in\II_+}p_k(\log|d_k|-\log a_k)}
\]
over $\Delta_r^0$ is $\sigma$, and is attained at $(p_1^*,\dots,p_r^*)$ (where we define $p_k^*:=0$ for $k\in\II_0$).
\end{lemma}

\begin{proof}
We could use the method of Lagrange multipliers, but the following argument is more direct. Let $q_k:=(|d_k|/a_k)^\sigma$, and observe that the inequality $G(p_1,\dots,p_r)\leq\sigma$ is equivalent to $\sum p_k\log p_k\geq \sum p_k\log q_k$ (the summations being over $\II_+$). But
\[
\sum p_k\log p_k-\sum p_k\log q_k=\sum p_k\log\left(\frac{p_k}{q_k}\right)\geq 0,
\]
since we recognize the last expression as the relative entropy (or Kullback-Leibler divergence) of the probability vector $(p_1,\dots,p_r)$ relative to the probability vector $(q_1,\dots,q_r)$, which is always nonnegative (an easy consequence of Jensen's inequality). Moreover, equality obtains when $p_k=q_k$ for all $k\in\II_+$.
\end{proof}

\begin{proposition} \label{prop:duality-b}
Assume $|d_k|<a_k$ for every $k$, and define
\[
h(\alpha):=\max\left\{\frac{(\alpha-1)\sum p_k\log p_k}{\sum p_k(\log|d_k|-\log a_k)}: \mathbf{p}\in\Delta_r^0,\ \sum p_k(\log|d_k|-\alpha\log a_k)\leq 0\right\},
\]
for $\alpha>1$, where the summations are over $k\in\II_+$. Then
\[
h(\alpha)=\begin{cases}
\sigma(\alpha-1) & \mbox{if $1\leq \alpha\leq\alpha_0$},\\
\beta^*(\alpha) & \mbox{if $\alpha_0\leq\alpha\leq\alpha_{\max}$}.
\end{cases}
\]
\end{proposition}

\begin{proof}
This follows from Proposition \ref{prop:duality} and Lemma \ref{lem:constrained-max}, by noting that $\sum p_k^*(\log|d_k|-\alpha\log a_k)<0$ if and only if $\alpha<\alpha_0$.
\end{proof}

\section{Examples} \label{sec:examples}

The first example, a generalization of Example \ref{ex:Takagi}, illustrates how to use Theorem \ref{thm:main} in combination with Propositions \ref{prop:duality} and \ref{prop:duality-b}. For more examples along these lines, see \cite{Allaart}.

\begin{example}[Skew Takagi function] \label{ex:skew-Takagi}
{\rm
Fix numbers $a\in(0,1/2)$, $h>0$ and $0<d<1$, and consider the self-affine function $\phi$ from \eqref{eq:functional-equation} with parameters $r=2$, $a_1=1-a_2=a$, $c_1=h=-c_2$, and $d_1=d_2=d$. We can view $\phi$ as a skewed version of the Takagi function from Example \ref{ex:Takagi}, with the tent map $\varphi$ replaced by the triangle with vertices $(0,0)$, $(1,0)$ and $(a,h)$; see Figure \ref{fig:skew-Takagi}. If $d\geq a$, we are in the situation of Theorem \ref{thm:main}\,({\em a}) and the H\"older spectrum of $\phi$ is given by the multifractal formalism. Here Proposition \ref{prop:duality} allows us to calculate $\beta^*(\alpha)$ very explicitly. Solving simultaneously the equations $p_1+p_2=1$ and $\sum_{k=1}^2 p_k(\log d_k-\alpha\log a_k)=0$, we obtain
\[
p_1=p:=\frac{\frac{\log d}{\alpha}-\log(1-a)}{\log a-\log(1-a)}, \qquad p_2=1-p,
\]
and then
\begin{equation}
\beta^*(\alpha)=H(p_1,p_2)=\frac{\alpha}{\log d}\left[p\log p+(1-p)\log(1-p)\right].
\label{eq:skew-Takagi-spectrum}
\end{equation}

Now assume that $d<a$ instead. We must then determine the set $\Lambda$. Using the characterization \eqref{eq:Lambda-for-two-maps}, we see that $\Lambda=\emptyset$ if and only if $d=a(1-a)$. Assuming this is not the case, we are in the situation of Theorem \ref{thm:main}\,({\em b}). With $\sigma$ the unique solution of $(d/a)^\sigma+(d/(1-a))^\sigma=1$, we can calculate 
\[
\alpha_0=\frac{d^{-\sigma}\log d}{a^{-\sigma}\log a+(1-a)^{-\sigma}\log(1-a)},
\]
and then $D(\alpha)$ is as given in \eqref{eq:special-spectrum}, with $\beta^*(\alpha)$ again given by \eqref{eq:skew-Takagi-spectrum}.
}
\end{example}

\begin{figure}
\begin{center}
\ \ \quad \epsfig{file=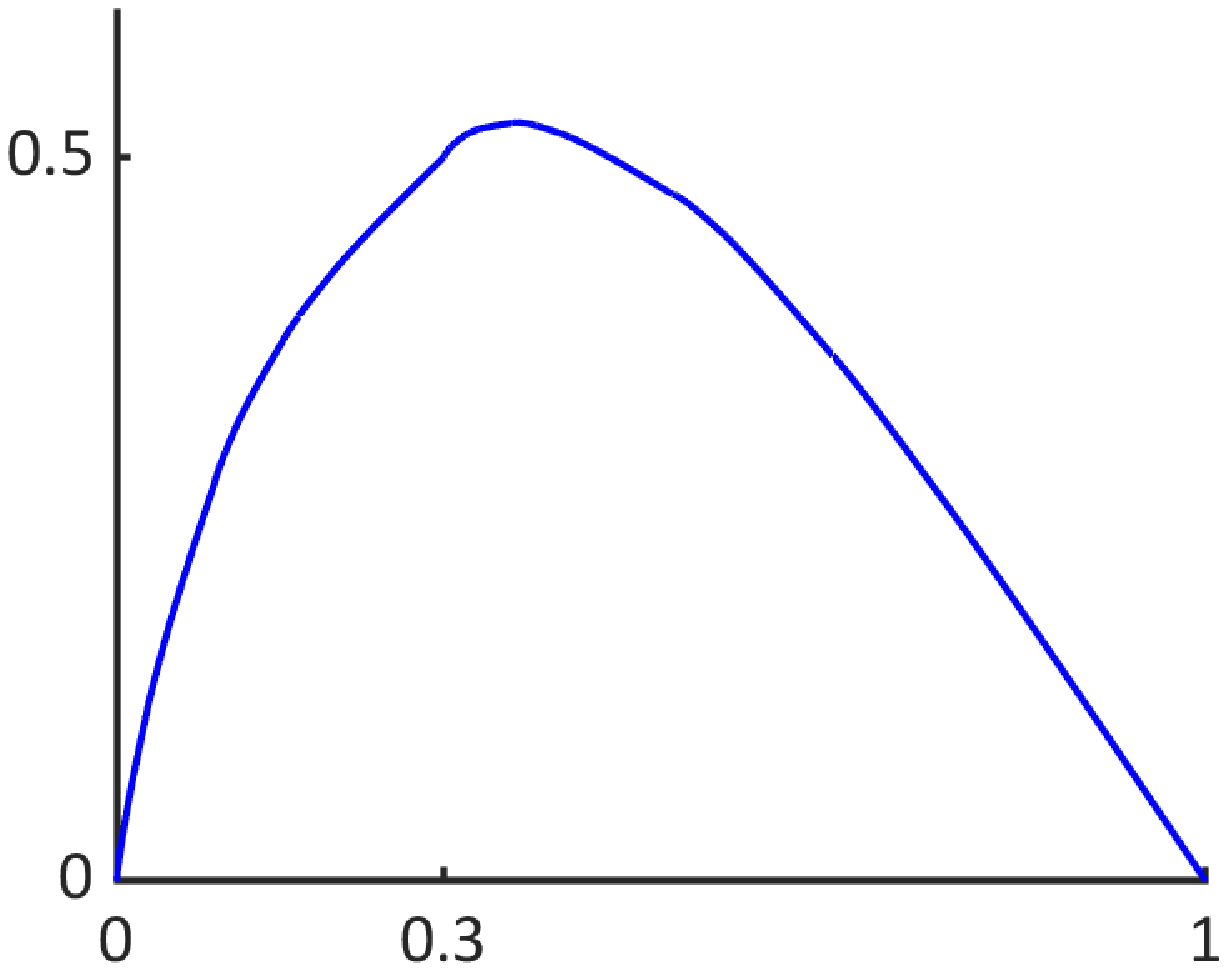, height=.25\textheight, width=.4\textwidth} \qquad\quad
\epsfig{file=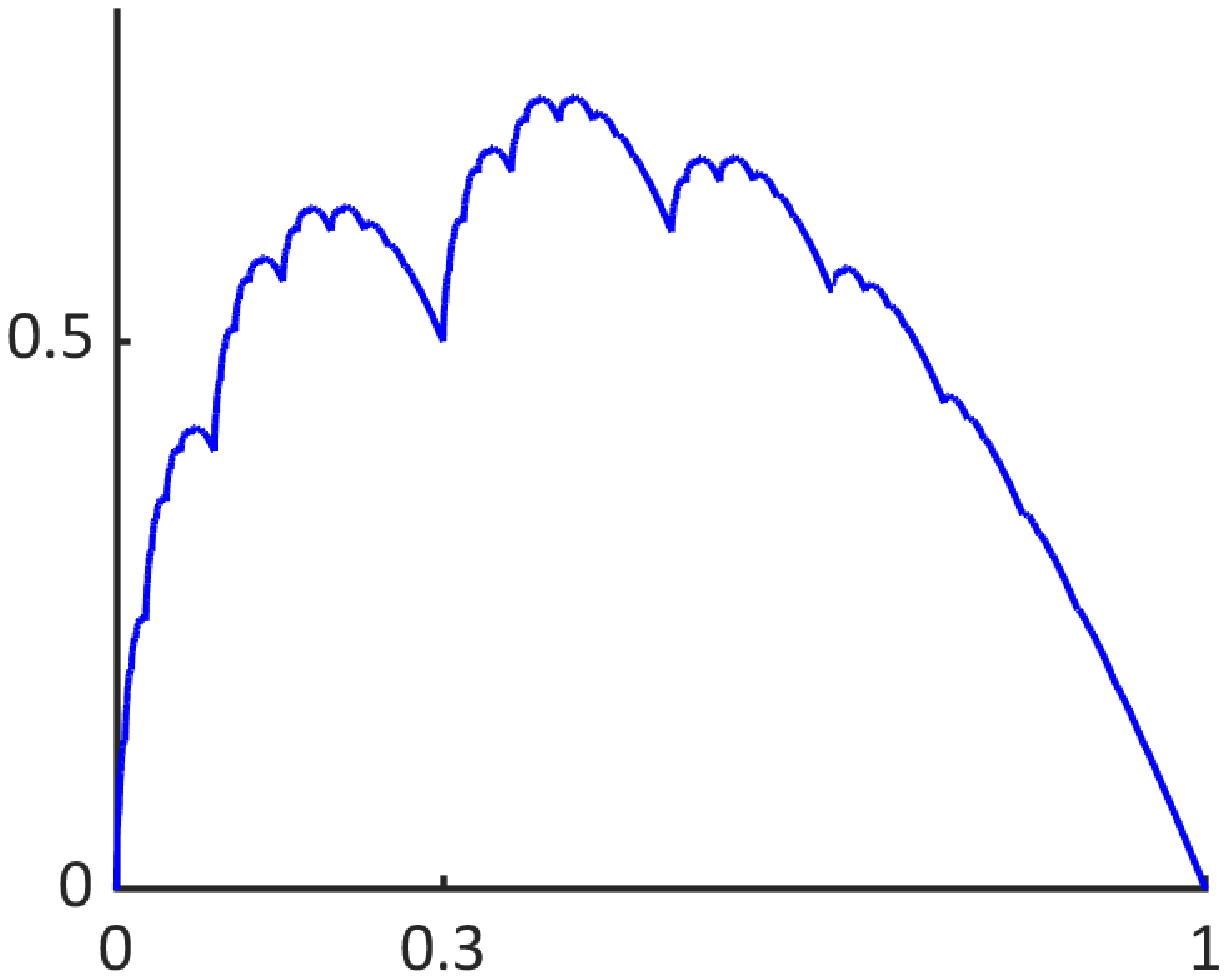, height=.25\textheight, width=.4\textwidth} \qquad\quad
\caption{The skew Takagi function from Example \ref{ex:skew-Takagi} with $a=0.3$, $h=0.5$ and $d=0.25$ (left), resp. $d=0.5$ (right)}
\label{fig:skew-Takagi}
\end{center}
\end{figure}

The next example shows that certain self-affine functions can be obtained as indefinite integrals of other self-affine functions.

\begin{example}
{\rm
Consider the case when $c_k=0$ for every $k$, so $\phi(a_k x+b_k)=d_k\phi(x)+e_k$ for $k=1,2,\dots,r$. Let $\psi(x):=\int_0^x \phi(t)dt$. A straightforward calculation shows that $\psi$ is again self-affine, satisfying
\[
\psi(a_k x+b_k)=\tilde{c}_k x+\tilde{d}_k \psi(x)+\tilde{e}_k, \qquad x\in [0,1], \quad k=1,2,\dots,r,
\]
where $\tilde{c}_k:=a_k e_k$, $\tilde{d}_k:=a_k d_k$ and $\tilde{e}_k=\int_0^{b_k}\phi(t)dt=\psi(b_k)$. It should be clear from the definition of $\psi(x)$ that $\psi$ is differentiable everywhere on $(0,1)$, and its multifractal spectrum is just that of $\phi$ shifted one unit to the right. Indeed, although $|\tilde{d}_k|<a_k$ for all $k$, it is not difficult to verify using \eqref{eq:parameter-relations} that $\Lambda=\emptyset$, so Theorem \ref{thm:main}\,({\em a}) applies equally to $\phi$ and $\psi$.

Vice versa, if $\phi$ satisfies \eqref{eq:functional-equation} with $|d_k|<a_k$ for each $k$ and $\Lambda=\emptyset$, then $\phi'$ exists everywhere and is again self-affine, with parameters $\hat{a}_k=a_k$, $\hat{b}_k=b_k$, $\hat{c}_k=0$, $\hat{d}_k=d_k/a_k$ and $\hat{e}_k=c_k/a_k$.
}
\end{example}

\section{The exact H\"older exponent of $\phi$} \label{sec:exact-Holder-exponent}

We first introduce some additional notation. Say $\phi\in \CC^\alpha_+(\xi)$ (resp.~$\phi\in \CC^\alpha_-(\xi)$) if there is a constant $C$ and a polynomial $P$ of degree less than $\alpha$ such that
\[
|\phi(x)-P(x)|\leq C|x-\xi|^\alpha \qquad\mbox{for all $x>\xi$ (resp.~$x<\xi$)}.
\]
Define the {\em right and left pointwise H\"older exponents} of $\phi$ at $\xi$, respectively, by
\[
\alpha_\phi^+(\xi):=\sup\{\alpha>0: \phi\in \CC^\alpha_+(\xi)\}, \qquad \alpha_\phi^-(\xi):=\sup\{\alpha>0: \phi\in \CC^\alpha_-(\xi)\}.
\]

We aim to give an exact expression for $\alpha_\phi^+(\xi)$ for all but countably many points; a formula for $\alpha_\phi^-(\xi)$ can be obtained similarly. Precisely, we will omit the points from the set 
$$\mathcal{T}:=\{S_{k_1}\circ\cdots\circ S_{k_n}(0),S_{k_1}\circ\cdots\circ S_{k_n}(1): n\in\NN,\ k_1,\dots,k_n\in\{1,2,\dots,r\}\},$$ 
where $S_k(x):=a_k x+b_k$. (These are the $r$-adic rational points in case $a_k=1/r$ for each $k$.) A precise statement requires the following notation. Let
\begin{equation}
K_1:=\begin{cases}
0 & \mbox{if $d_1=0$ or $d_r=0$},\\
\left(\frac{\log a_r}{\log a_1}\right)\log|d_1|-\log|d_r| & \mbox{otherwise},
\end{cases}
\label{eq:K1}
\end{equation}
and
\begin{equation}
K_2:=\begin{cases}
0 & \mbox{if $d_r=0$},\\
\log a_r-\log|d_r| & \mbox{if $d_r\neq 0$}.
\end{cases}
\label{eq:K2}
\end{equation}
Fix a point $\xi\in(0,1)\backslash\mathcal{T}$. There is a unique sequence $(k_1,k_2,\dots)$ of indices from $\{1,2,\dots,r\}$, called the {\em coding} of $\xi$, such that $\xi=\lim_{n\to\infty}S_{k_1}\circ S_{k_2}\circ\dots\circ S_{k_n}(0)$.
Set
\[
s_{n,k}(\xi):=\#\{i: 1\leq i\leq n, k_i=k\}, \qquad n\in\NN, \quad k=1,2,\dots,r,
\]
and let
\[
L_n(\xi):=\max\{j\in\NN: k_{n-j+1}=k_{n-j+2}=\dots=k_n=r\},
\]
or $L_n(\xi)=0$ in case $k_n<r$. We now define two indicators,
\[
\chi_n(\xi):=\begin{cases}
1 & \mbox{if $k_{n-L_n(\xi)}+1\in\II_+$},\\
0 & \mbox{otherwise},
\end{cases}
\]
and
\[
\zeta_n(\xi):=\begin{cases}
1 & \mbox{if $k_{n-L_n(\xi)}\in\Lambda$},\\
0 & \mbox{otherwise}.
\end{cases}
\]
Assuming $k_i\in\II_+$ for all $i$, we define
\begin{gather*}
\gamma_0=\gamma_0(\xi):=\liminf_{n\to\infty}\frac{\sum_{k=1}^r s_{n,k}(\xi)\log|d_k|}{\sum_{k=1}^r s_{n,k}(\xi)\log a_k},\\
\gamma_1=\gamma_1(\xi):=\liminf_{n\to\infty}\frac{\sum_{k=1}^r s_{n,k}(\xi)\log|d_k|+K_1\chi_n(\xi)L_n(\xi)}{\sum_{k=1}^r s_{n,k}(\xi)\log a_k},
\end{gather*}
and
\[
\gamma_2=\gamma_2(\xi):=\liminf_{n\to\infty}\frac{\sum_{k=1}^r s_{n,k}(\xi)\log|d_k|+K_2\zeta_n(\xi)L_n(\xi)}{\sum_{k=1}^r s_{n,k}(\xi)\log a_k}.
\]
Finally, set
\[
\gamma=\gamma(\xi):=\min\{\gamma_0,\gamma_1,\gamma_2\}.
\]

\begin{theorem} \label{thm:exact-Holder-exponent}
Assume $a_k>0$ for all $k$, and let $\xi\in(0,1)\backslash\mathcal{T}$ with coding $(k_1,k_2,\dots)$. Assume also that $\phi$ is not a polynomial.
Then
\[
\alpha_\phi^+(\xi)=\begin{cases}
\gamma(\xi) & \mbox{if $k_i\in\II_+$ for all $i$},\\
+\infty & \mbox{otherwise}.
\end{cases}
\]
Moreover, if $\gamma(\xi)>1$, then the right derivative of $\phi$ at $\xi$ is given by
\begin{equation}
\phi_+'(\xi)=\sum_{m=1}^\infty \frac{c_{k_m}}{d_{k_m}}\prod_{k=1}^r \left(\frac{d_k}{a_k}\right)^{s_{m,k}}.
\label{eq:right-derivative}
\end{equation}
\end{theorem}

Note that the denominator in the expressions for $\gamma_1$ and $\gamma_2$ is negative. Hence, exceptionally large values of $L_n(\xi)$ can reduce the pointwise H\"older exponent of $\phi$ at $\xi$ from the ``default" value $\gamma_0$ when $K_1$ or $K_2$ is positive. On the other hand, if $L_n(\xi)=o(n)$, then we simply have $\alpha_\phi^+(\xi)=\gamma_0(\xi)$.

\begin{remark} \label{rem:left-Holder-exponent}
{\rm
We can similarly give an expression for $\alpha_\phi^-(\xi)$. Rather than stating a formal theorem, we briefly indicate how to modify the one above. First, in the definitions of $K_1$, $K_2$ and $L_n$, switch the roles of the digits $1$ and $r$. The definition of $\chi_n$ should be changed to $\chi_n(\xi)=1$ if $k_{n-L_n(\xi)}-1\in\II_+$ and $0$ otherwise; and the definition of $\zeta_n$ should be changed to $\zeta_n(\xi)=1$ if $k_{n-L_n(\xi)}-1\in\Lambda$ and $0$ otherwise. Then the expression in the theorem gives $\alpha_\phi^-(\xi)$, and if this number is greater than 1, then the left derivative $\phi_-'(\xi)$ is given by the right hand side of \eqref{eq:right-derivative}.
}
\end{remark}

\begin{remark}
{\rm
It is possible, in principle, to give exact expressions for $\alpha_\phi^+(\xi)$ also in cases where $a_k<0$ for some $k$. However, such statements would be even more complicated than the ones given, for instance, in \cite[Theorem 6.1]{Allaart}. We therefore do not pursue this here.
}
\end{remark}

Before proving Theorem \ref{thm:exact-Holder-exponent}, we point out some relationships between the quantities $\gamma_0,\gamma_1$ and $\gamma_2$. 
Clearly, $\gamma_1\leq\gamma_0$ when $K_1\geq 0$, and $\gamma_2\leq\gamma_0$ when $K_2\geq 0$. More subtle are the following inequalities.

\begin{lemma} \label{lem:gammas}
Assume $\xi\not\in\mathcal{T}$ and $k_i\in\II_+$ for all $i$. Then $\gamma_2\geq\min\{\gamma_0,1\}$. Moreover, if $\gamma_0>1$ and $\limsup_{n\to\infty} L_n(\xi)/n<1$, then $\gamma_2>1$.
\end{lemma}

\begin{proof}
Throughout, we suppress the dependence on $\xi$. Observe that
\begin{align}
\begin{split}
\gamma_2&\geq \liminf_{n\to\infty} \frac{\sum s_{n,k}\log|d_k|+K_2 L_n}{\sum s_{n,k}\log a_k}\\
&=\liminf_{n\to\infty} \frac{\sum s_{n-L_n,k}\log|d_k|+L_n\log a_r}{\sum s_{n-L_n,k}\log a_k+L_n\log a_r}\\
&=\liminf_{n\to\infty} \frac{\sum s_{n-L_n,k}|\log|d_k||+L_n|\log a_r|}{\sum s_{n-L_n,k}|\log a_k|+L_n|\log a_r|}.
\end{split}
\label{eq:gamma_2-lower-estimate}
\end{align}
Since
\[
\liminf_{n\to\infty} \frac{\sum s_{n-L_n,k}\log|d_k|}{\sum s_{n-L_n,k}\log a_k}\geq\liminf_{n\to\infty} \frac{\sum s_{n,k}\log|d_k|}{\sum s_{n,k}\log a_k}=\gamma_0,
\]
it follows that $\gamma_2\geq 1$ when $\gamma_0>1$, and $\gamma_2\geq\gamma_0$ when $\gamma_0\leq 1$.

To prove the second statement, suppose $\gamma_0>1$ and $\limsup L_n/n<1$. Then there is a constant $M$ such that $L_n/(n-L_n)<M$ for every $n$. 
There also is a $\delta>0$ such that, for all large enough $n$, $\sum s_{n-L_n,k}|\log|d_k||>(1+\delta)\sum s_{n-L_n,k}|\log a_k|$.
Starting from \eqref{eq:gamma_2-lower-estimate}, we then have for all sufficiently large $n$,
\begin{align*}
\frac{\sum s_{n-L_n,k}|\log|d_k||+L_n|\log a_r|}{\sum s_{n-L_n,k}|\log a_k|+L_n|\log a_r|}
&>\frac{(1+\delta)\sum s_{n-L_n,k}|\log a_k|+L_n|\log a_r|}{\sum s_{n-L_n,k}|\log a_k|+L_n|\log a_r|}\\
&\geq \frac{(1+\delta)\sum s_{n-L_n,k}|\log a_k|+M(n-L_n)|\log a_r|}{\sum s_{n-L_n,k}|\log a_k|+M(n-L_n)|\log a_r|}\\
&=1+\frac{\frac{1}{n-L_n}\delta\sum s_{n-L_n,k}|\log a_k|+M|\log a_r|}{\frac{1}{n-L_n}\sum s_{n-L_n,k}|\log a_k|+M|\log a_r|}\\
&\geq 1+\frac{\delta\min_k|\log a_k|+M|\log a_r|}{\max_k |\log a_k|+M|\log a_r|},
\end{align*}
independent of $n$. Hence, $\gamma_2>1$.
\end{proof}

\begin{corollary} \label{cor:equal-gammas}
If $|d_r|\leq a_r$ and $\gamma_0(\xi)\leq 1$, then $\gamma_0(\xi)=\gamma_2(\xi)$.
\end{corollary}

\begin{proof}
By Lemma \ref{lem:gammas}, $\gamma_2(\xi)\geq\gamma_0(\xi)$. On the other hand, $|d_r|\leq a_r$ implies $K_2\geq 0$, so $\gamma_2(\xi)\leq\gamma_0(\xi)$ directly from the definition of $\gamma_2(\xi)$.
\end{proof}

\begin{corollary} \label{cor:alpha-range}
Let $\xi\in(0,1)\backslash\mathcal{T}$. 
\begin{enumerate}[(i)]
\item Under the hypotheses of Theorem \ref{thm:main}\,({a}), $\alpha_\phi^+(\xi)\in [\alpha_{\min},\alpha_{\max}]\cup\{\infty\}$.
\item Under the hypotheses of Theorem \ref{thm:main}\,({b}), $\alpha_\phi^+(\xi)\in [1,\alpha_{\max}]\cup\{\infty\}$.
\end{enumerate}
\end{corollary}

\begin{proof}
If $k_i(\xi)\in\II_0$ for some $i$, then $\alpha_\phi^+(\xi)=\infty$, so assume $k_i(\xi)\in\II_+$ for all $i$. Clearly $\gamma_0(\xi)\in[\alpha_{\min},\alpha_{\max}]$, so $\alpha_\phi^+(\xi)\leq\alpha_{\max}$. As shown in \cite[Lemma 6.3]{Allaart}, $\gamma_1(\xi)\geq\alpha_{\min}$ also. If $|d_k|\geq a_k$ for some $k$, then $\alpha_{\min}\leq 1$ and so $\gamma_2(\xi)\geq\min\{\gamma_0(\xi),1\}\geq\alpha_{\min}$ by Lemma \ref{lem:gammas}. If $\Lambda=\emptyset$, then $\gamma_2(\xi)=\gamma_0(\xi)$ because $\zeta_n(\xi)=0$ for all $n$. Finally, if $|d_k|<a_k$ for all $k$, then $\alpha_{\min}>1$ and so $\gamma_2\geq 1$ by Lemma \ref{lem:gammas}. Thus, the corollary follows from Theorem \ref{thm:exact-Holder-exponent}.
\end{proof}

The proof of Theorem \ref{thm:exact-Holder-exponent} is quite long and technical. We split it in two parts: First we prove the lower bound (Steps 1 and 2), and then, after introducing some useful facts about divided differences, we prove the upper bound (Steps 3 and 4).

We shall use the following terminology and notation. By a {\em basic interval} of {\em order} $n$ we shall mean an interval of the form $I=S_{k_1}\circ\dots\circ S_{k_n}([0,1])$, where $k_1,\dots,k_n\in\{1,2,\dots,r\}$. For any interval $I$, $|I|$ will denote its length. Let $I_{n,j}: j=1,2,\dots,r^n$ denote the basic intervals of order $n$, enumerated in order from left to right. For $\xi\not\in\mathcal{T}$, there is for each $n$ a unique $j$ such that $\xi\in I_{n,j}$; denote this interval $I_{n,j}$ by $I_n(\xi)$. If $I_{n,j}=S_{k_1}\circ\dots\circ S_{k_n}([0,1])$ and $x,\xi\in I_{n,j}$, there are unique points $t_n,\tau_n\in [0,1]$ such that $S_{k_1}\circ\dots\circ S_{k_n}(t_n)=x$ and $S_{k_1}\circ\dots\circ S_{k_n}(\tau_n)=\xi$. In that case, we have the explicit expression (see \cite[p.~135]{Dubuc})
\begin{equation}
\phi(x)-\phi(\xi)=(t_n-\tau_n)\sum_{m=1}^n c_{k_m}\prod_{i=1}^{m-1}d_{k_i}\prod_{i=m+1}^n a_{k_i}+\big(\phi(t_n)-\phi(\tau_n)\big)\prod_{m=1}^n d_{k_m}.
\label{eq:difference-expression}
\end{equation}

Finally, we use the short hand notation
\[
a_{\max}:=\max_{k\in\mathcal{I}} a_k, \qquad a_{\min}:=\min_{k\in\mathcal{I}} a_k.
\]

\begin{proof}[Proof of Theorem \ref{thm:exact-Holder-exponent} (lower bound)]
Here we show that $\alpha_\phi^+(\xi)\geq \gamma(\xi)$. We assume throughout the proof that $K_1>0$ and $K_2>0$; the proof in the other cases is simpler and shorter. Note that our assumption implies that $\gamma=\min\{\gamma_1,\gamma_2\}$, and furthermore
\begin{equation}
|d_1|>0, \qquad |d_r|>0, \qquad a_r>|d_r|, \qquad\mbox{and}\qquad \frac{\log|d_1|}{\log a_1}<\frac{\log|d_r|}{\log a_r}.
\label{eq:assumed-inequalities}
\end{equation}
Fix $\xi$ with coding $(k_1,k_2,\dots)$. The case when $k_i\in\II_0$ for some $i$ is trivial, so we assume that $k_i\in\II_+$ for all $i$.
Let $\gamma=\gamma(\xi)$. 

\bigskip
{\em Step 1.} We assume that $\gamma\leq 1$ and show that $\phi\in \CC_+^\alpha(\xi)$ for every $\alpha<\gamma$. 

Fix $\alpha<\gamma$ and note that in particular, $\alpha<\gamma_1(\xi)\leq\gamma_0(\xi)$. Take a point $x>\xi$ with $x-\xi\leq a_{\min}^2$, and let $n$ be the largest integer such that
\begin{equation}
x-\xi\leq a_{\min}|I_n(\xi)|=a_{\min} a_{k_1}\dots a_{k_n}.
\label{eq:choice-of-n}
\end{equation}
Then $x-\xi>a_{k_1}\dots a_{k_{n+1}}a_{\min}\geq a_{k_1}\dots a_{k_n}a_{\min}^2$, so there is an absolute constant $A>0$ such that $A^{-1}<(x-\xi)/a_{k_1}\dots a_{k_n}<A$. We write this as $x-\xi\asymp a_{k_1}\dots a_{k_n}$.

\bigskip
{\em Case 1.} If $x\in I_n(\xi)$, we have simply, using \eqref{eq:difference-expression},
\begin{equation*}
\frac{|\phi(x)-\phi(\xi)|}{(x-\xi)^\alpha}\asymp \left|(t_n-\tau_n)\sum_{m=1}^n \frac{c_{k_m}}{d_{k_m}}\prod_{i=1}^m \frac{d_{k_i}}{a_{k_i}}\prod_{i=1}^n a_{k_i}^{1-\alpha} + \big(\phi(t_n)-\phi(\tau_n)\big)\prod_{m=1}^n \frac{d_{k_m}}{a_{k_m}^\alpha}\right|
\end{equation*}
for some $t_n,\tau_n\in[0,1]$. Since $|t_n-\tau_n|$, $c_{k_m}/d_{k_m}$ and $|\phi(t_n)-\phi(\tau_n)|$ are all bounded, it follows that there is a constant $C_1$ such that
\begin{equation}
\frac{|\phi(x)-\phi(\xi)|}{(x-\xi)^\alpha}\leq C_1\left[\sum_{m=1}^n \prod_{k=1}^r\left(\frac{|d_k|}{a_k^\alpha}\right)^{s_{m,k}}\prod_{k=1}^r a_k^{(1-\alpha)(s_{n,k}-s_{m,k})}+\prod_{k=1}^r \left(\frac{|d_k|}{a_k^\alpha}\right)^{s_{n,k}}\right].
\label{eq:monster-difference-estimate}
\end{equation}
Now, since $\alpha<\gamma\leq\gamma_0(\xi)$, 
there is a number $\eps>0$ and an integer $N$ such that $\alpha+\eps<1$ and $\sum s_{n,k}\log|d_k|<(\alpha+\eps)\sum s_{n,k}\log a_k$ for all $n>N$, and so
\begin{equation}
\prod_{k=1}^r\left(\frac{|d_k|}{a_k^\alpha}\right)^{s_{n,k}}<\prod_{k=1}^r a_k^{\eps s_{n,k}}\leq a_{\max}^{\eps n} \qquad\forall n>N.
\label{eq:product-estimate}
\end{equation}
This shows that the last product in square brackets in \eqref{eq:monster-difference-estimate} tends to zero. The summation in \eqref{eq:monster-difference-estimate} we split in two parts. Observe first that
\[
\prod_{k=1}^r a_k^{(1-\alpha)(s_{n,k}-s_{m,k})} \leq a_{\max}^{(1-\alpha)(n-m)}.
\]
Since $\alpha<1$, it follows at once that
\[
\sum_{m=1}^N \prod_{k=1}^r\left(\frac{|d_k|}{a_k^\alpha}\right)^{s_{m,k}}a_{\max}^{(1-\alpha)(n-m)} \to 0 \qquad\mbox{as $n\to\infty$}.
\]
For the remaining terms we can apply \eqref{eq:product-estimate}, obtaining
\begin{align*}
\sum_{m=N+1}^n \prod_{k=1}^r\left(\frac{|d_k|}{a_k^\alpha}\right)^{s_{m,k}}a_{\max}^{(1-\alpha)(n-m)} &< \sum_{m=N+1}^n a_{\max}^{\eps m}a_{\max}^{(1-\alpha)(n-m)}\\
&\leq a_{\max}^{(1-\alpha)n}\cdot\frac{a_{\max}^{(\alpha+\eps-1)(n+1)}}{a_{\max}^{\alpha+\eps-1}-1}\\
&=\frac{a_{\max}^{\alpha+\eps-1}}{a_{\max}^{\alpha+\eps-1}-1}\cdot a_{\max}^{\eps n}
\to 0 \qquad\mbox{as $n\to\infty$}.
\end{align*}
As a result, the upper estimate for
$|\phi(x)-\phi(\xi)|/(x-\xi)^\alpha$ in \eqref{eq:monster-difference-estimate}, which holds when $x\in I_n(\xi)$, tends to zero as $n\to\infty$. 

\bigskip
{\em Case 2.} Suppose now that $x\not\in I_n(\xi)$. 
Let $l:=L_n(\xi)$. Note that $|I_n(\xi)|=a_{k_1}\cdots a_{k_n}=a_{k_1}\cdots a_{k_{n-l}}a_r^l$. Let $p$ be the largest integer such that
\begin{equation}
a_{k_1}\cdots a_{k_{n-l-1}}a_{k_{n-l}+1}a_1^{l+p}\geq x-\xi.
\label{eq:choice-of-p}
\end{equation}
Observe that $l+p\geq 0$ in view of \eqref{eq:choice-of-n}. Hence, there is an index $j'$ such that $|I_{n+p,j'}|=a_{k_1}\cdots a_{k_{n-l-1}}a_{k_{n-l}+1}a_1^{l+p}$, and it follows that the interval $I_{n+p,j'}$ is adjacent to $I_n(\xi)$. By \eqref{eq:choice-of-p}, $x\in I_{n+p,j'}$ and so $I_{n+p,j'}=I_{n+p}(x)$. Observe that our choice of $p$ implies
\begin{equation}
a_1^{l+p} \asymp a_r^l,
\label{eq:a-power-asymp}
\end{equation}
and then also
\begin{equation}
|d_1|^{l+p} \asymp |d_1|^{l\log a_r/\log a_1}.
\label{eq:d-power-asymp}
\end{equation}
The interval $I_{n+p}(x)$ has coding $(k_1',\dots,k_{n+p}')$, where $k_i'=k_i$ for $1\leq i<n-l$, $k_{n-l}'=k_{n-l}+1$, and $k_i'=1$ for $n-l<i\leq n+p$. Let $\xi_n\in\mathcal{T}$ be the point joining $I_n(\xi)$ and $I_{n+p}(x)$. We have already shown that $\big(\phi(\xi_n)-\phi(\xi)\big)/(x-\xi)^\alpha\to 0$, so it remains to establish that
\begin{equation}
\frac{\phi(x)-\phi(\xi_n)}{(x-\xi)^\alpha}\to 0.
\end{equation}
(Recall that $x\downarrow \xi$ and $n$ is determined by $x$ through \eqref{eq:choice-of-n}.) By \eqref{eq:difference-expression}, we have
\begin{equation}
\phi(x)-\phi(\xi_n)=t_n\sum_{m=1}^{n+p} c_{k_m'}\prod_{i=1}^{m-1}d_{k_i'}\prod_{i=m+1}^{n+p} a_{k_i'}+\big(\phi(t_n)-\phi(0)\big)\prod_{m=1}^{n+p} d_{k_m'},
\end{equation}
for some point $t_n\in[0,1]$. Hence, for some constant $C_2$,
\begin{equation}
|\phi(x)-\phi(\xi_n)|\leq C_2\left[\sum_{m=1}^{n+p}\prod_{k=1}^r|d_k|^{s_{m,k}'}\prod_{k=1}^r a_k^{s_{n+p,k}'-s_{m,k}'}+\prod_{k=1}^r |d_k|^{s_{n+p,k}'}\right],
\label{eq:K6-estimate}
\end{equation}
where $s_{m,k}':=\#\{i\leq m: k_i'=k\}$ for $1\leq m\leq n+p$. We first rewrite
\begin{equation*}
\prod_{k=1}^r |d_k|^{s_{n+p,k}'}=\prod_{k=1}^r|d_k|^{s_{n,k}} \cdot \frac{|d_{k_{n-l}+1}|}{|d_{k_{n-l}}|} \cdot \frac{|d_1|^{l+p}}{|d_r|^l}.
\end{equation*}
If $d_{k_{n-l}+1}=0$, then $\prod_{k=1}^r |d_k|^{s_{n+p,k}'}=0$ and we are done with this term. Otherwise, $\chi_n(\xi)=1$, and \eqref{eq:d-power-asymp} gives
\begin{equation}
\prod_{k=1}^r |d_k|^{s_{n+p,k}'} \asymp \prod_{k=1}^r|d_k|^{s_{n,k}} \cdot \left(\frac{|d_1|^{\log a_r/\log a_1}}{|d_r|}\right)^l=e^{K_1 L_n}\prod_{k=1}^r|d_k|^{s_{n,k}},
\label{eq:prime-product}
\end{equation}
using \eqref{eq:K1}. In this case, since $\alpha<\gamma_1(\xi)$, there exist $\eps>0$ and $N\in\NN$ such that $\sum s_{n,k}\log|d_k|+K_1\chi_n L_n<(\alpha+\eps)\sum s_{n,k}\log a_k$ for all $n>N$. It follows that, analogously to \eqref{eq:product-estimate},
\begin{equation}
e^{K_1\chi_n L_n} \prod_{k=1}^r\left(\frac{|d_k|}{a_k^\alpha}\right)^{s_{n,k}}< a_{\max}^{\eps n} \qquad\forall n>N.
\label{eq:K1-product-estimate}
\end{equation}
Hence by \eqref{eq:prime-product},
\begin{equation}
(x-\xi)^{-\alpha}\prod_{k=1}^r |d_k|^{s_{n+p,k}'}\asymp e^{K_1\chi_n L_n} \prod_{k=1}^r\left(\frac{|d_k|}{a_k^\alpha}\right)^{s_{n,k}}\to 0.
\label{eq:prime-prod-asymp}
\end{equation}
The summation in \eqref{eq:K6-estimate} we divide in two parts: 
\[
\sum_{m=1}^{n+p}\prod_{k=1}^r|d_k|^{s_{m,k}'}\prod_{k=1}^r a_k^{s_{n+p,k}'-s_{m,k}'}=\sum_{m=1}^{n-l-1}+\sum_{m=n-l}^{n+p}=:S_1+S_2.
\]
First, we rewrite
\[
S_1=\sum_{m=1}^{n-l-1}\prod_{k=1}^r|d_k|^{s_{m,k}}\prod_{k=1}^r a_k^{s_{n,k}-s_{m,k}}\prod_{k=1}^r a_k^{s_{n+p,k}'-s_{n,k}}.
\]
By \eqref{eq:a-power-asymp},
\[
\prod_{k=1}^r a_k^{s_{n+p,k}'-s_{n,k}}=\frac{a_{k_{n-l}+1}}{a_{k_{n-l}}}\cdot\frac{a_1^{l+p}}{a_r^l} \asymp 1,
\]
and therefore,
\[
S_1 \asymp \sum_{m=1}^{n-l-1}\prod_{k=1}^r|d_k|^{s_{m,k}}\prod_{k=1}^r a_k^{s_{n,k}-s_{m,k}}.
\]
Thus $(x-\xi)^{-\alpha}S_1\to 0$ as in Case 1. The second summation we rewrite as
\[
S_2=\sum_{m=n-l}^{n+p} \prod_{k=1}^r |d_k|^{s_{n+p,k}'} \left(\frac{a_1}{|d_1|}\right)^{n+p-m}=\prod_{k=1}^r |d_k|^{s_{n+p,k}'} \cdot \sum_{j=0}^{l+p}\left(\frac{a_1}{|d_1|}\right)^j.
\]
We now consider three cases:

(a) If $a_1<|d_1|$, then $S_2 \asymp \prod_{k=1}^r |d_k|^{s_{n+p,k}'}$, so $(x-\xi)^{-\alpha}S_2\to 0$ by \eqref{eq:prime-prod-asymp}.

(b) If $a_1=|d_1|$, then $S_2=(l+p+1)\prod_{k=1}^r |d_k|^{s_{n+p,k}'}$. By \eqref{eq:a-power-asymp}, there are constants $C_3$ and $C_4$ such that $l+p<C_3 l+C_4$, so $l+p+1\leq 2C_3 n$ for all large enough $n$, since $l\leq n$. Thus, by \eqref{eq:prime-prod-asymp} and \eqref{eq:K1-product-estimate},
$(x-\xi)^{-\alpha}S_2\to 0$.

(c) If $a_1>|d_1|$, then the largest term in $S_2$ is the one with $m=n-l$, but this term differs at most by a multiplicative constant from the last term in $S_1$. So in this case, too, $(x-\xi)^{-\alpha}S_2\to 0$. This completes Step 1.

\bigskip

{\em Step 2.} We assume that $\gamma>1$ and show that the series
\begin{equation}
D(\xi):=\sum_{m=1}^\infty \frac{c_{k_m}}{d_{k_m}}\prod_{k=1}^r \left(\frac{d_k}{a_k}\right)^{s_{m,k}}
\label{eq:the-derivative-of-phi}
\end{equation}
converges absolutely, and
\begin{equation}
\frac{|\phi(x)-\phi(\xi)-D(\xi)(x-\xi)|}{(x-\xi)^\alpha}\to 0 \qquad \mbox{for $1<\alpha<\gamma$}.
\label{eq:Holder-exponent-bigger-than-one}
\end{equation}
It then follows that $\phi\in \CC_+^\alpha(\xi)$ for every $\alpha<\gamma$, and $\phi$ is right differentiable with $\phi_+'(\xi)=D(\xi)$.

First, since $\gamma>1$, \eqref{eq:product-estimate} holds with $\alpha=1$ for some $\eps>0$ and some $N\in\NN$. Since $|c_{k_m}/d_{k_m}|$ is uniformly bounded, we conclude that the series in \eqref{eq:the-derivative-of-phi} converges absolutely.

For $x>\xi$, let $n$ again be the largest integer satisfying \eqref{eq:choice-of-n}. 

\bigskip
{\em Case 1.} Assume first that $x\in I_n(\xi)$. Then there are numbers $t_n,\tau_n$ in $[0,1]$ such that
\begin{equation*}
\phi(x)-\phi(\xi)=(x-\xi)\sum_{m=1}^n \frac{c_{k_m}}{d_{k_m}}\prod_{k=1}^r \left(\frac{d_k}{a_k}\right)^{s_{m,k}}+\left(\phi(t_n)-\phi(\tau_n)\right)\prod_{k=1}^r d_k^{s_{n,k}},
\end{equation*}
so that
\begin{align}
\begin{split}
&\frac{|\phi(x)-\phi(\xi)-D(\xi)(x-\xi)|}{(x-\xi)^\alpha}=\\
&\qquad\qquad =(x-\xi)^{1-\alpha}\sum_{m=n+1}^\infty \frac{c_{k_m}}{d_{k_m}}\prod_{k=1}^r \left(\frac{d_k}{a_k}\right)^{s_{m,k}}+\frac{\phi(t_n)-\phi(\tau_n)}{(x-\xi)^\alpha}\prod_{k=1}^r d_k^{s_{n,k}}.
\end{split}
\label{eq:exact-difference-quotient}
\end{align}
There exist $\eps>0$ and $N\in\NN$ such that \eqref{eq:product-estimate} holds. So for $n>N$,
\begin{align*}
&\left|(x-\xi)^{1-\alpha}\sum_{m=n+1}^\infty \frac{c_{k_m}}{d_{k_m}}\prod_{k=1}^r \left(\frac{d_k}{a_k}\right)^{s_{m,k}}\right|
\leq C_1(x-\xi)^{1-\alpha} \sum_{m=n+1}^\infty \prod_{k=1}^r \left(\frac{|d_k|}{a_k}\right)^{s_{m,k}}\\
&\qquad\qquad \asymp \prod_{k=1}^r a_k^{(1-\alpha)s_{n,k}} \sum_{m=n+1}^\infty \prod_{k=1}^r \left(\frac{|d_k|}{a_k}\right)^{s_{m,k}}\\
&\qquad\qquad =\sum_{m=n+1}^\infty \prod_{k=1}^r\left(\frac{|d_k|}{a_k^\alpha}\right)^{s_{m,k}} \prod_{k=1}^r a_k^{(\alpha-1)(s_{m,k}-s_{n,k})}\\
&\qquad\qquad \leq \sum_{m=n+1}^\infty a_{\max}^{\eps m}a_{\max}^{(\alpha-1)(m-n)}
 =a_{\max}^{-(\alpha-1)n} \sum_{m=n+1}^\infty a_{\max}^{(\eps+\alpha-1)m}\\
&\qquad\qquad \asymp a_{\max}^{-(\alpha-1)n} a_{\max}^{(\eps+\alpha-1)n}
 =a_{\max}^{\eps n}\to 0.
\end{align*}
Again by \eqref{eq:product-estimate}, 
\[
(x-\xi)^{-\alpha}\prod_{k=1}^r |d|_k^{s_{n,k}} \asymp \prod_{k=1}^r\left(\frac{|d_k|}{a_k^\alpha}\right)^{s_{n,k}}\to 0.
\]
Combining these last two results with \eqref{eq:exact-difference-quotient} yields \eqref{eq:Holder-exponent-bigger-than-one} for $x\in I_n(\xi)$.

\bigskip
{\em Case 2.} Suppose now that $x\not\in I_n(\xi)$. Let $l:=L_n(\xi)$. We define the integer $p$, the adjoining interval $I_{n+p,j'}$ to the right of $I_n(\xi)$ and the connecting point $\xi_n\in\mathcal{T}$ as in Case 2 of Step 1, and note that \eqref{eq:a-power-asymp} and \eqref{eq:d-power-asymp} hold. By \eqref{eq:difference-expression}, there are points $\tau_n,t_n\in[0,1]$ such that
\begin{equation}
\phi(\xi_n)-\phi(\xi)=(\xi_n-\xi)\sum_{m=1}^n \frac{c_{k_m}}{d_{k_m}}\prod_{k=1}^r\left(\frac{d_k}{a_k}\right)^{s_{m,k}}
+\big(\phi(1)-\phi(\tau_n)\big)\prod_{k=1}^r d_k^{s_{n,k}},
\label{eq:difference-first-half}
\end{equation}
and
\begin{equation}
\phi(x)-\phi(\xi_n)=(x-\xi_n)\sum_{m=1}^{n+p}\frac{c_{k_m'}}{d_{k_m}'}\prod_{k=1}^r \left(\frac{d_k}{a_k}\right)^{s_{m,k}'}
+\big(\phi(t_n)-\phi(0)\big)\prod_{k=1}^r d_k^{s_{n+p,k}'},
\label{eq:difference-second-half}
\end{equation}
where $k_i'$ and $s_{m,k}'$ are defined as in Step 1. As in \eqref{eq:prime-prod-asymp}, $(x-\xi)^{-\alpha}\prod_{k=1}^r |d_k|^{s_{n+p,k}'}\to 0$ and, more straightforwardly, $(x-\xi)^{-\alpha}\prod_{k=1}^r |d_k|^{s_{n,k}}\to 0$, so the remainder terms in \eqref{eq:difference-first-half} and \eqref{eq:difference-second-half} are of no concern. Putting $\nu:=n-l$ for brevity, we can write the summation in \eqref{eq:difference-second-half} as
\begin{equation}
\sum_{m=1}^{n+p}\frac{c_{k_m'}}{d_{k_m}'}\prod_{k=1}^r \left(\frac{d_k}{a_k}\right)^{s_{m,k}'}= \sum_{m=1}^n \frac{c_{k_m}}{d_{k_m}}\prod_{k=1}^r\left(\frac{d_k}{a_k}\right)^{s_{m,k}}+B_n \prod_{k=1}^r \left(\frac{d_k}{a_k}\right)^{s_{\nu,k}},
\end{equation}
where
\begin{align}
\begin{split}
B_n:=&\left(\frac{c_{k_{\nu}+1}a_{k_{\nu}}}{a_{k_{\nu}+1}d_{k_{\nu}}}-\frac{c_{k_{\nu}}}{d_{k_{\nu}}}\right)+\frac{c_1}{d_1}\sum_{m=\nu+1}^{n+p}\frac{d_{k_{\nu}+1}a_{k_{\nu}}}{a_{k_{\nu}+1}d_{k_{\nu}}}\left(\frac{d_1}{a_1}\right)^{m-\nu}\\
&\qquad\qquad\qquad\qquad\qquad\qquad\qquad-\frac{c_r}{d_r}\sum_{m=\nu+1}^n\left(\frac{d_r}{a_r}\right)^{m-\nu}.
\end{split}
\label{eq:def-of-Bn}
\end{align}
Thus,
\[
\phi(x)-\phi(\xi)=(x-\xi)\sum_{m=1}^n \frac{c_{k_m}}{d_{k_m}}\prod_{k=1}^r\left(\frac{d_k}{a_k}\right)^{s_{m,k}} +(x-\xi_n)B_n \prod_{k=1}^r \left(\frac{d_k}{a_k}\right)^{s_{\nu,k}}+R_1+R_2,
\]
where $R_1$ and $R_2$ denote the remainder terms in \eqref{eq:difference-first-half} and \eqref{eq:difference-second-half}, respectively. This gives
\begin{align*}
\frac{|\phi(x)-\phi(\xi)-D(\xi)(x-\xi)|}{(x-\xi)^\alpha} &\leq (x-\xi)^{1-\alpha}\sum_{m=n+1}^\infty \frac{c_{k_m}}{|d_{k_m}|}\prod_{k=1}^r\left(\frac{|d_k|}{a_k}\right)^{s_{m,k}}\\
&\qquad+(x-\xi)^{1-\alpha}|B_n| \prod_{k=1}^r \left(\frac{|d_k|}{a_k}\right)^{s_{\nu,k}}+\frac{|R_1|+|R_2|}{(x-\xi)^\alpha}.
\end{align*}
It has already been established that the first and last terms tend so zero, so we focus on the term involving $B_n$.

Recall from our assumptions at the beginning of the proof that $|d_r|<a_r$, so the second summation in \eqref{eq:def-of-Bn} is bounded. We consider three subcases:

(i) If $|d_1|>a_1$, then $K_1>K_2$. If $d_{k_\nu+1}\neq 0$, then $\chi_n(\xi)=1$ and the dominant term in the first summation in $B_n$ is the one with $m=n+p$, which is of order $(|d_1|/a_1)^{n+p-\nu}=(|d_1|/a_1)^{l+p}$. So in this case,
\begin{align*}
(x-\xi)^{1-\alpha}|B_n| \prod_{k=1}^r \left(\frac{|d_k|}{a_k}\right)^{s_{\nu,k}}
&\asymp \prod_{k=1}^r a_k^{(1-\alpha)s_{n,k}}\prod_{k=1}^r \left(\frac{|d_k|}{a_k}\right)^{s_{n,k}}\left(\frac{a_r}{|d_r|}\right)^l \left(\frac{|d_1|}{a_1}\right)^{l+p}\\
&\asymp \prod_{k=1}^r \left(\frac{|d_k|}{a_k^\alpha}\right)^{s_{n,k}}\left(\frac{|d_1|^{\log a_r/\log a_1}}{|d_r|}\right)^l\\
&=e^{K_1\chi_n(\xi)L_n(\xi)}\prod_{k=1}^r\left(\frac{|d_k|}{a_k^\alpha}\right)^{s_{n,k}} \to 0
\end{align*}
as in Step 1, where we used \eqref{eq:a-power-asymp} and \eqref{eq:d-power-asymp} in the second step, and the convergence to zero follows from \eqref{eq:K1-product-estimate}. If, on the other hand, $d_{k_\nu+1}=0$, then $B_n$ simplifies to
\begin{equation} 
B_n=\frac{c_{k_{\nu}+1}a_{k_{\nu}}}{a_{k_{\nu}+1}d_{k_{\nu}}}-\frac{c_{k_{\nu}}}{d_{k_{\nu}}}-\frac{c_r}{a_r-d_r}\left\{1-\left(\frac{d_r}{a_r}\right)^l\right\}.
\label{eq:B_n-simplified}
\end{equation}
At this point, there are two possibilities: 

(a) $k_\nu\not\in\Lambda$. Then $B_n$ simplifies further to 
\[
B_n=\frac{c_r}{a_r-d_r}\left(\frac{d_r}{a_r}\right)^l
\]
by definition of $\Lambda$, so that \eqref{eq:product-estimate} gives
\begin{align*}
(x-\xi)^{1-\alpha}|B_n| \prod_{k=1}^r \left(\frac{|d_k|}{a_k}\right)^{s_{\nu,k}} 
&\asymp \prod_{k=1}^r a_k^{(1-\alpha)s_{n,k}} \left(\frac{|d_r|}{a_r}\right)^l \prod_{k=1}^r \left(\frac{|d_k|}{a_k}\right)^{s_{\nu,k}}\\
&=\prod_{k=1}^r \left(\frac{|d_k|}{a_k^\alpha}\right)^{s_{n,k}} \to 0.
\end{align*}

(b) $k_\nu\in\Lambda$. Then $\zeta_n(\xi)=1$ and $B_n\asymp 1$, so that
\begin{align}
\begin{split}
(x-\xi)^{1-\alpha}|B_n| \prod_{k=1}^r \left(\frac{|d_k|}{a_k}\right)^{s_{\nu,k}} 
&\asymp \left(\frac{a_r}{|d_r|}\right)^l \prod_{k=1}^r \left(\frac{|d_k|}{a_k^\alpha}\right)^{s_{n,k}}\\
&=e^{K_2\zeta_n(\xi)L_n(\xi)}\prod_{k=1}^r \left(\frac{|d_k|}{a_k^\alpha}\right)^{s_{n,k}} \to 0,
\end{split}
\label{eq:K2-exponential}
\end{align}
since $\alpha<\gamma_2$, where we used the obvious analogy to \eqref{eq:K1-product-estimate}.

(ii) Suppose next that $|d_1|=a_1$. If $d_{k_\nu+1}=0$, the situation is as in case (i). If $d_{k_\nu+1}\neq 0$, then we have $|B_n|\asymp l+p$, and hence
\[
(x-\xi)^{1-\alpha}|B_n| \prod_{k=1}^r \left(\frac{|d_k|}{a_k}\right)^{s_{\nu,k}} 
\asymp (l+p)e^{K_1\chi_n(\xi)L_n(\xi)}\prod_{k=1}^r\left(\frac{|d_k|}{a_k^\alpha}\right)^{s_{n,k}} \to 0
\]
by \eqref{eq:K1-product-estimate}, since $l+p=O(n)$ (see Step 1).

(iii) Suppose finally that $|d_1|<a_1$. Then $K_2>K_1$. Summing the finite geometric series in \eqref{eq:def-of-Bn} gives
\begin{align}
\begin{split}
B_n=\frac{c_{k_{\nu}+1}a_{k_{\nu}}}{a_{k_{\nu}+1}d_{k_{\nu}}}&-\frac{c_{k_{\nu}}}{d_{k_{\nu}}}+\frac{c_1}{a_1-d_1}\cdot\frac{a_{k_\nu}d_{k_\nu+1}}{d_{k_\nu}a_{k_\nu+1}}\left\{1-\left(\frac{d_1}{a_1}\right)^{l+p}\right\}\\
&\qquad\qquad -\frac{c_r}{a_r-d_r}\left\{1-\left(\frac{d_r}{a_r}\right)^l\right\}.
\end{split}
\label{eq:B_n-summed}
\end{align}
Note that $B_n$ is bounded. If $k_\nu\in\Lambda$, then $\zeta_n(\xi)=1$ and we have \eqref{eq:K2-exponential}. Suppose $k_\nu\not\in\Lambda$. Then $B_n$ simplifies to
\[
B_n=\frac{c_r}{a_r-d_r}\left(\frac{d_r}{a_r}\right)^l-\frac{c_1}{a_1-d_1}\cdot\frac{a_{k_\nu}d_{k_\nu+1}}{d_{k_\nu}a_{k_\nu+1}}\left(\frac{d_1}{a_1}\right)^{l+p}.
\]
As for the first term, we see at once that
\[
(x-\xi)^{1-\alpha}\left(\frac{|d_r|}{a_r}\right)^l \prod_{k=1}^r \left(\frac{|d_k|}{a_k}\right)^{s_{\nu,k}} 
\asymp \prod_{k=1}^r\left(\frac{|d_k|}{a_k^\alpha}\right)\to 0.
\]
The second term vanishes when $d_{k_\nu+1}=0$. If $d_{k_\nu+1}\neq 0$, then $\chi_n(\xi)=1$, and we obtain
\[
(x-\xi)^{1-\alpha}\left(\frac{|d_1|}{a_1}\right)^{l+p} \prod_{k=1}^r \left(\frac{|d_k|}{a_k}\right)^{s_{\nu,k}} 
\asymp e^{K_1\chi_n(\xi)L_n(\xi)}\prod_{k=1}^r\left(\frac{|d_k|}{a_k^\alpha}\right)^{s_{n,k}}\to 0
\]
as in case (i) above. This completes Step 2, and the proof of the lower bound.
\end{proof}

The proof of the upper bound in Theorem \ref{thm:exact-Holder-exponent} uses the technique of divided differences. We briefly review the definition and basic properties. For a function $f$ and a finite list of distinct points $x_0,x_1,\dots,x_n$, the divided difference $f[x_0,x_1,\dots,x_n]$ is defined inductively as follows:
\[
f[x_i]:=f(x_i), \qquad i=0,1,\dots,n,
\]
and
\begin{align*}
f[x_i,\dots,x_{i+k}]&:=\frac{f[x_{i+1},\dots,x_{i+k}]-f[x_i,\dots,x_{i+k-1}]}{x_{i+k}-x_i},\\
&\qquad\qquad\qquad\qquad i=0,1,\dots,n-k, \quad k=1,2,\dots,n.
\end{align*}
Divided differences have some of the same properties as higher order derivatives. They are linear in $f$ and satisfy a mean value theorem. For the purposes of this article, the most important properties are the following:

\begin{lemma} \label{lem:divided-differences}
\begin{enumerate}[(i)]
\item If $f$ is a polynomial of degree less than $n$, then $f[x_0,x_1,\dots,x_n]=0$ for every choice of points $x_0,x_1,\dots,x_n$.
\item If $f$ is not a polynomial on $[a,b]$, then for every $n\in\NN$ there are points $x_0,x_1,\dots,x_n$ in $[a,b]$ such that $f[x_0,x_1,\dots,x_n]\neq 0$.
\end{enumerate}
\end{lemma}

The next lemma, whose proof can be found in \cite[Lemma 12]{Dubuc}, is crucial for proving the upper bound on $\alpha_\phi^+(\xi)$.

\begin{lemma}[Dubuc] \label{lem:oscillation-bound}
Let $f:[a,b]\to\RR$ and let $x_0<x_1<\dots<x_N$ be an increasing sequence of $N+1$ distinct points in $[a,b]$. Then for any $x\in[a,b]$ there is an index $k\in\{0,1,\dots,N\}$ such that
\[
|f(x)-f(x_k)|\geq N!\left(\frac{\delta}{2}\right)^N|f[x_0,x_1,\dots,x_N]|,
\]
where $\delta:=\min\{x_k-x_{k-1}: k=1,\dots,N\}$.
\end{lemma}

\begin{proof}[Proof of Theorem \ref{thm:exact-Holder-exponent} (upper bound)]
Here we show that $\alpha_\phi^+(\xi)\leq \gamma(\xi)$. We assume again that $K_1>0$ and $K_2>0$, so that $\gamma=\min\{\gamma_1,\gamma_2\}$ and the inequalities \eqref{eq:assumed-inequalities} hold. 

\bigskip
{\em Step 3.} We assume that $\gamma<1$ and show that $\phi\not\in \CC_+^\alpha(\xi)$ for any $\alpha>\gamma$. Note that $\gamma<1$ implies $\gamma=\gamma_1(\xi)$ by Lemma \ref{lem:gammas}. We begin by showing that $\phi\not\in \CC_+^\alpha(\xi)$ if $\alpha>\gamma_0$.

Since $\phi$ is not a polynomial, it is not linear on $[0,1]$ and so there exists $t_0\in(0,1)$ such that $\phi[0,t_0,1]\neq 0$. Set $g_n:=S_{k_1}\circ S_{k_2}\circ\dots\circ S_{k_n}$ for $n\in\NN$, and denote $\xi_n:=g_n(0)$, $\xi_n':=g_n(1)$ and $t_n:=g_n(t_0)$. From \eqref{eq:difference-expression}, we obtain
\[
\phi[\xi_n,t_n]=\sum_{m=1}^n \frac{c_{k_m}}{d_{k_m}}\prod_{k=1}^r \left(\frac{d_k}{a_k}\right)^{s_{k,m}}+\phi[0,t_0]\prod_{k=1}^r \left(\frac{d_k}{a_k}\right)^{s_{k,n}}
\]
and
\[
\phi[t_n,\xi_n']=\sum_{m=1}^n \frac{c_{k_m}}{d_{k_m}}\prod_{k=1}^r \left(\frac{d_k}{a_k}\right)^{s_{k,m}}+\phi[t_0,1]\prod_{k=1}^r \left(\frac{d_k}{a_k}\right)^{s_{k,n}},
\]
so that
\[
\phi[\xi_n,t_n,\xi_n']=\phi[0,t_0,1]\prod_{k=1}^r \left(\frac{d_k}{a_k^2}\right)^{s_{k,n}}.
\]
Note that $\xi\in(\xi_n,\xi_n')$.
Thus, by Lemma \ref{lem:oscillation-bound}, there is a point $w_n\in\{\xi_n,t_n,\xi_n'\}$ such that $|\phi(w_n)-\phi(\xi)|\geq C_1\prod_{k=1}^r|d_k|^{s_{k,n}}$ with $C_1:=\min\{t_0,1-t_0\}|\phi[0,t_0,1]|/2>0$. Hence,
\[
\frac{|\phi(w_n)-\phi(\xi)|}{|w_n-\xi|^\alpha}\geq \frac{|\phi(w_n)-\phi(\xi)|}{(\xi_n'-\xi_n)^\alpha}\geq C_1\prod_{k=1}^r \left(\frac{d_k}{a_k^\alpha}\right)^{s_{k,n}},
\]
so by the counterpart of \eqref{eq:product-estimate},
\[
\limsup_{x\searrow \xi} \frac{\phi(x)-\phi(\xi)}{(x-\xi)^\alpha}\geq \limsup_{n\to\infty} C_1 \prod_{k=1}^r \left(\frac{d_k}{a_k^\alpha}\right)^{s_{k,n}}=\infty
\]
for any $\alpha>\gamma_0$.

If $\chi_n(\xi)=0$ for all but finitely many $n$, then $\gamma=\gamma_0$ and we are done. Assume therefore that $\chi_n(\xi)=1$ for infinitely many $n$. There is then an increasing subsequence $(n_i)$ such that $\chi_{n_i}(\xi)=1$ for each $i$, and
\[
\lim_{i\to\infty}\frac{\sum_{k=1}^r s_{n_i,k}(\xi)\log|d_k|+K_1\chi_{n_i}(\xi)L_{n_i}(\xi)}{\sum_{k=1}^r s_{n_i,k}(\xi)\log a_k}=\gamma_1(\xi).
\]
Take $n=n_i$ and let $l:=L_n(\xi)$. Let $j$ be the integer such that $I_n(\xi)=I_{n,j}$. We may assume also that $l\geq 1$ (otherwise we simply scratch this $n$ from the subsequence $(n_i)$). Then the next basic interval $I_{n,j+1}$ has length $a_{k_1}\dots a_{k_{n-l-1}}a_{k_{n-l}+1}a_1^l$. Note that
\[
a_{k_1}\dots a_{k_{n-l-1}}a_{k_{n-l}+1}\geq a_{\min} a_{k_1}\dots a_{k_{n-l}}\geq a_{\min} a_{k_1}\dots a_{k_n}=a_{\min}|I_n(\xi)|.
\]
Let $p$ be the largest integer such that $a_{k_1}\dots a_{k_{n-l-1}}a_{k_{n-l}+1}a_1^{l+p}\geq a_{\min}|I_n(\xi)|$.
Then $l+p\geq 0$, so there is an integer $j'$ such that the basic interval $I_{n+p,j'}$ has length $|I_{n+p,j'}|=a_{k_1}\cdots a_{k_{n-l-1}}a_{k_{n-l}+1}a_1^{l+p}$ and is adjacent to (and to the right of) $I_n(\xi)$. Write $I_{n+p,j'}=[\xi_n',\xi_n'']$. Let $g_n'$ be the affine map satisfying $g_n'(0)=\xi_n'$ and $g_n'(1)=\xi_n''$. Put $t_n':=g_n'(t_0)$. Applying Lemma \ref{lem:oscillation-bound} to the interval $I_{n+p,j'}$ we see that there is a point $w_n\in\{\xi_n',t_n',\xi_n''\}$ such that $|\phi(w_n)-\phi(\xi_n')|\geq C_1\prod_{k=1}^r|d_k|^{s_{k,n+p}'}$, with $C_1$ as above. Then for at least one $x\in\{w_n,\xi_n'\}$,
\[
|\phi(x)-\phi(\xi)|\geq (C_1/2)\prod_{k=1}^r|d_k|^{s_{k,n+p}'}.
\]
Now, since $\alpha>\gamma_1(\xi)$, there is an $\eps>0$ and an integer $N$ such that if $n\in(n_i)$ and $n>N$, then $\sum s_{n,k}\log|d_k|+K_1 l>(\alpha-\eps)\sum s_{n,k}\log a_k$. This gives the dual to \eqref{eq:prime-prod-asymp}, namely
\begin{equation}
\prod_{k=1}^r |d_k|^{s_{n+p,k}'} \asymp e^{K_1 l}\prod_{k=1}^r|d_k|^{s_{n,k}}>\left(\prod_{k=1}^r a_k^{s_{n,k}}\right)^{\alpha-\eps}.
\label{eq:prime-prod-asymp-dual}
\end{equation}
Since
\begin{equation}
x-\xi\leq \xi_n''-\xi_n=|I_n(\xi)|+|I_{n+p,j'}|\leq \left(1+\frac{a_{\min}}{a_1}\right)|I_n(\xi)|\leq 2\prod_{k=1}^r a_k^{s_{n,k}},
\label{eq:x-difference-estimate}
\end{equation}
it follows that
\[
\frac{|\phi(x)-\phi(\xi)|}{(x-\xi)^\alpha}\geq C_2\left(\prod_{k=1}^r a_k^{s_{n,k}}\right)^{-\eps}=C_2|I_n(\xi)|^{-\eps}\to\infty,
\]
for some constant $C_2>0$. Thus, $\phi\not\in \CC_+^\alpha(\xi)$. This completes Step 3.

\bigskip

{\em Step 4.} We assume that $\gamma\geq 1$ and show that $\phi\not\in \CC_+^\alpha(\xi)$ for any $\alpha>\gamma$.

\medskip
{\em Case 1.} Suppose first that $\gamma=\gamma_2<\gamma_1$. Take $\alpha\in(\gamma,\gamma_1)$. From the work in Step 2 it follows that
\[
\frac{|\phi(x)-\phi(\xi)-D(\xi)(x-\xi)|}{(x-\xi)^\alpha}\to 0 \qquad \mbox{for $x\in I_n(\xi)$, as $n\to\infty$},
\]
with $D(\xi)$ defined as in \eqref{eq:the-derivative-of-phi}. Hence, if there is a polynomial $P(x)$ such that \eqref{eq:pointwise-Holder-property} holds, it must be the case that $P(x)=\phi(\xi)+D(\xi)(x-\xi)$. We show that this leads to a contradiction.

Let $(n_i)$ be an increasing sequence of positive integers such that
\[
\lim_{i\to\infty} \frac{\sum_{k=1}^r s_{n_i,k}(\xi)\log|d_k|+K_2\zeta_{n_i}(\xi)L_{n_i}(\xi)}{\sum_{k=1}^r s_{n_i,k}(\xi)\log a_k}=\gamma_2. 
\]
Since $\gamma_2<\gamma_0$, we may assume that for each $n\in(n_i)$, $\zeta_n(\xi)=1$ and so $k_{n-L_n(\xi)}\in\Lambda$. Moreover, $\limsup_{i\to\infty}L_{n_i}(\xi)/n_i>0$, so we may assume that $L_{n_i}(\xi)>l_0$ for all $i$, for some suitable number $l_0$ to be chosen later. Take $n\in(n_i)$, and put $l:=L_n(\xi)$ and $\nu:=n-l$. Choose $p$ and $j'$ as in Step 3, and write $I_{n+p,j'}=[\xi_n',\xi_n'']$. Recall that $I_{n+p,j'}$ is a basic interval adjacent to $I_n(\xi)$.

In Step 2 we derived for $x\in I_{n+p,j'}$ the expression
\begin{align*}
\phi(x)-\phi(\xi)-D(\xi)(x-\xi) &=(x-\xi)\sum_{m=n+1}^\infty \frac{c_{k_m}}{d_{k_m}}\prod_{k=1}^r\left(\frac{d_k}{a_k}\right)^{s_{m,k}}\\
&\qquad+(x-\xi_n')B_n \prod_{k=1}^r \left(\frac{d_k}{a_k}\right)^{s_{\nu,k}}+R_1+R_2,
\end{align*}
with $B_n$ as defined by \eqref{eq:def-of-Bn}.
Since $\alpha<\gamma_1$, we have
\[
(x-\xi)^{-\alpha}\left[(x-\xi)\sum_{m=n+1}^\infty \frac{c_{k_m}}{d_{k_m}}\prod_{k=1}^r\left(\frac{d_k}{a_k}\right)^{s_{m,k}}+R_1+R_2\right]\to 0, \qquad x\in I_{n+p,j'},
\]
as in Step 2. We show that the remaining term,
\[
(x-\xi)^{-\alpha}(x-\xi_n')B_n \prod_{k=1}^r \left(\frac{d_k}{a_k}\right)^{s_{\nu,k}},
\]
is unbounded if we take $x=\xi_n''$. Observe that by the choice of $p$, $\xi_n''-\xi_n'=|I_{n+p,j'}|\geq a_{\min}|I_n(\xi)|$, and similarly, $\xi_n''-\xi\leq |I_n(\xi)|+|I_{n+p,j'}|\leq 2|I_n(\xi)|$ as in \eqref{eq:x-difference-estimate}. Hence
\[
(\xi_n''-\xi)^{-\alpha}(\xi_n''-\xi_n')\geq C_3|I_n(\xi)|^{1-\alpha},
\]
where $C_3:=2^{-\alpha}a_{\min}>0$. We claim that $B_n$ is bounded away from zero when $k_\nu\in\Lambda$. To show this, we distinguish two cases:

(i) $|d_1|<a_1$. Here we can write $B_n$ as in \eqref{eq:B_n-summed}. We first show that $l+p$ is large when $l$ is large. From the choice of $p$ it follows that
\[
a_1^{l+p+1}<\frac{a_{\min}a_{k_1}\dots a_{k_n}}{a_{k_1}\dots a_{k_{\nu-1}}a_{k_\nu+1}}=\frac{a_{\min}a_{k_\nu}\dots a_{k_n}}{a_{k_\nu+1}}
\leq a_{k_\nu}\dots a_{k_n}\leq a_{\max}^{l+1},
\]
so that
\begin{equation}
l+p+1>\frac{\log a_{\max}}{\log a_1}(l+1).
\label{eq:l-plus-p-bound}
\end{equation}
Let
\[
C_\Lambda:=\min\left\{\left|\frac{c_{k+1}a_k}{a_{k+1}c_k}-\frac{c_k}{d_k}+\frac{c_1}{a_1-d_1} \frac{a_k d_{k+1}}{d_k a_{k+1}}-\frac{c_r}{a_r-d_r}\right|: k\in\Lambda\cap \II_+\right\}.
\]
Then $C_\Lambda>0$ by definition of $\Lambda$. Thus, since $\limsup L_{n_i}(\xi)/n_i>0$, we can assume in view of \eqref{eq:B_n-summed} and \eqref{eq:l-plus-p-bound} that $l$ is large enough so that $|B_n|\geq \frac12 C_\Lambda>0$.

(ii) $|d_1|\geq a_1$. Then $K_1\geq K_2$, so our assumption that $\gamma_2<\gamma_1$ implies that $d_{k_{n-L_n(\xi)}+1}=0$ for all sufficiently large $n$. So we may assume $d_{k_\nu+1}=0$. Then (see the note following \eqref{eq:Lambda}) $k_\nu\in\Lambda$ means that
\[
\frac{c_{k_\nu+1}a_{k_\nu}}{a_{k_\nu+1}c_{k_\nu}}-\frac{c_{k_\nu}}{d_{k_\nu}}-\frac{c_r}{a_r-d_r}\neq 0,
\]
and of course there are only finitely many possible values for this expression. Denote the smallest of these (in absolute value) by $C_\Lambda'$. At the same time, $B_n$ now simplifies to \eqref{eq:B_n-simplified}. Thus, we can assume $l$ is large enough so that $|B_n|\geq \frac12 |C_\Lambda'|>0$. This proves our claim.

It now remains to show that
\[
|I_n(\xi)|^{1-\alpha}\prod_{k=1}^r \left(\frac{d_k}{a_k}\right)^{s_{\nu,k}}\asymp \prod_{k=1}^r a_k^{(1-\alpha)s_{n,k}}\prod_{k=1}^r \left(\frac{d_k}{a_k}\right)^{s_{\nu,k}}\to\infty
\]
along $n\in(n_i)$. But this follows since
\begin{equation*}
\prod_{k=1}^r a_k^{(1-\alpha)s_{n,k}}\prod_{k=1}^r \left(\frac{d_k}{a_k}\right)^{s_{\nu,k}}
=\left(\frac{a_r}{|d_r|}\right)^l \prod_{k=1}^r \left(\frac{|d_k|}{a_k^\alpha}\right)^{s_{n,k}}
=e^{K_2 l}\prod_{k=1}^r \left(\frac{|d_k|}{a_k^\alpha}\right)^{s_{n,k}},
\end{equation*}
as in \eqref{eq:K2-exponential}, and $\alpha>\gamma_2$. We conclude that $\phi\not\in \CC_+^\alpha(\xi)$.

\bigskip
{\em Case 2.} Suppose on the other hand that $\gamma=\gamma_1$. Take $\alpha>\gamma$, and let $N$ be the smallest integer greater than $\gamma$. Since $\phi$ is not a polynomial, there is by Lemma \ref{lem:divided-differences}(ii) a partition $0=t_0<t_1<\dots<t_N=1$ of $[0,1]$ such that $\phi[t_0,t_1,\dots,t_N]\neq 0$. Let $g_n:=S_{k_1}\circ S_{k_2}\circ\dots\circ S_{k_n}$ for $n\in\NN$. Let $P$ be a polynomial of degree less than $N$ such that $P(\xi)=\phi(\xi)$, and set $R(x):=\phi(x)-P(x)$.

Assume first that $\gamma_1=\gamma_0$. Let $\xi_n^{(j)}:=g_n(t_j)$ for $j=0,1,\dots,N$. By Lemma \ref{lem:divided-differences}(i), $P[\xi_n^{(0)},\xi_n^{(1)},\dots,\xi_n^{(N)}]=0$, and so
\begin{equation}
R[\xi_n^{(0)},\xi_n^{(1)},\dots,\xi_n^{(N)}]=\phi[\xi_n^{(0)},\xi_n^{(1)},\dots,\xi_n^{(N)}]=\phi[t_0,t_1,\dots,t_N]\prod_{k=1}^r\left(\frac{d_k}{a_k^N}\right)^{s_{n,k}}.
\end{equation}
Applying Lemma \ref{lem:oscillation-bound} to the function $R$ and the partition $\xi_n^{(0)},\dots,\xi_n^{(N)}$, we see that there is a point $w_n\in\{\xi_n^{(0)},\dots,\xi_n^{(N)}\}$ such that
\[
|R(w_n)|\geq C_4\prod_{k=1}^r|d_k|^{s_{n,k}},
\]
where $C_4:=N!(\delta/2)^N|\phi[t_0,t_1,\dots,t_N]|>0$, with $\delta=\min_j (t_j-t_{j-1})$. Hence
\begin{equation*}
\limsup_{x\searrow\xi}\frac{|R(x)|}{(x-\xi)^\alpha} \geq \limsup_{n\to\infty} \frac{|R(w_n)|}{|I_n(\xi)|^\alpha}
\geq C_4\limsup_{n\to\infty} \prod_{k=1}^r\left(\frac{d_k}{a_k^\alpha}\right)^{s_{n,k}}=\infty,
\end{equation*}
since $\alpha>\gamma_0$.

Assume next that $\gamma_1<\gamma_0$. Then there is a subsequence $(n_i)$ such that for each $n\in(n_i)$, $\chi_n(\xi)=1$. Take such an $n$. We now proceed as in the second case of Step 3, choosing the integer $p$ and the interval $I_{n+p,j'}$ in the same way and letting $g_n'$ be the affine map which maps $[0,1]$ onto $I_{n+p,j'}$ without reflections. Now set $\xi_n^{(j)}:=g_n'(t_j)$ for $j=0,1,\dots,N$. This time we apply Lemma \ref{lem:oscillation-bound} to the function $R$ and the partition $\xi_n^{(0)},\dots,\xi_n^{(N)}$ of $I_{n+p,j'}$ to obtain a point $w_n\in \{\xi_n^{(0)},\dots,\xi_n^{(N)}\}$ such that
\[
|R(w_n)|\geq C_4\prod_{k=1}^r|d_k|^{s_{n+p,k}'}.
\]
Precisely as in Step 3 (see \eqref{eq:prime-prod-asymp-dual} and beyond), it now follows that
\[
\limsup_{n\to\infty}\frac{|R(w_n)|}{(w_n-\xi)^\alpha}=\infty.
\]
Thus, $\phi\not\in \CC_+^\alpha(\xi)$. This completes the proof.
\end{proof}

\begin{corollary} \label{cor:minimum}
If $\xi\not\in\TT$, then $\alpha_\phi(\xi)=\min\{\alpha_\phi^+(\xi),\alpha_\phi^-(\xi)\}$.
\end{corollary}

\begin{proof}
The proof of Theorem \ref{thm:exact-Holder-exponent} shows that the polynomial $P$ approximating $\phi(x)$ for $x>\xi$ is given by $P(x)=\phi(\xi)$ for $\alpha\leq 1$, and by $P(x)=\phi(\xi)+D(\xi)(x-\xi)$ for $\alpha>1$, where $D(\xi)$ is given by \eqref{eq:the-derivative-of-phi}. The proof shows only that $D(\xi)$ is the right derivative of $\phi$ at $\xi$. A completely analogous argument shows, however, that when $\alpha_\phi^-(\xi)>1$, $D(\xi)$ is also the left derivative of $\phi$ at $\xi$, and the polynomial approximating $\phi(x)$ for $x<\xi$ is the same as that for $x>\xi$. Hence, $\alpha_\phi(\xi)=\min\{\alpha_\phi^+(\xi),\alpha_\phi^-(\xi)\}$.
\end{proof}

\begin{remark} \label{rem:cut-points}
{\rm
The situation is different for points of $\TT$. Take $\xi\in\TT$; then $\xi$ has two codings $(k_i^-)=(k_1,\dots,k_{n_0},r,r,r,\dots)$ and $(k_i^+)=(k_1,\dots,k_{n_0}+1,1,1,1,\dots)$. It is straightforward to verify that $\alpha_\phi^+(\xi)=\log|d_1|/\log a_1$ and $\alpha_\phi^-(\xi)=\log|d_r|/\log a_r$. Let us assume that $|d_1|<a_1$ and $|d_r|<a_r$, so that $\alpha_\phi^+(\xi)>1$ and $\alpha_\phi^-(\xi)>1$. Then $\phi_+'(\xi)$ exists and is given by \eqref{eq:the-derivative-of-phi} applied to the coding $(k_i^+)$. Similarly, the left derivative $\phi_-'(\xi)$ is given by \eqref{eq:the-derivative-of-phi} applied to the coding $(k_i^-)$. Working out the two infinite series, one finds that $\phi_+'(\xi)=\phi_-'(\xi)$ if and only if $k_{n_0}\not\in\Lambda$. When that is the case, $\phi$ is differentiable at $\xi$ and $\alpha_\phi(\xi)=\min\{\alpha_\phi^+(\xi),\alpha_\phi^-(\xi)\}$. But if $k_{n_0}\in\Lambda$, then $\phi$ is not differentiable at $\xi$ and $\alpha_\phi(\xi)=1<\min\{\alpha_\phi^+(\xi),\alpha_\phi^-(\xi)\}$. 
}
\end{remark}

\section{Proof of the lower bound} \label{sec:lower-bound}

Here we assume first that we are in the setting of Theorem \ref{thm:main}\,({\em b}); that is, $|d_k|<a_k$ for all $k$, and $\Lambda\neq\emptyset$. 
Without loss of generality we may assume that there is an element $k^*$ of $\Lambda$ such that $d_{k^*}\neq 0$. (If this is not the case, then there is a $k^*\in\Lambda$ such that $d_{k^*+1}\neq 0$, and we reverse the roles of $\alpha_\phi^+$ and $\alpha_\phi^-$ in what follows; see Remark \ref{rem:left-Holder-exponent}.) Note that $k^*<r$.

We begin by constructing particular sets of sequences in $\{1,2,\dots,r\}^\NN$. Fix $\lambda\in(0,1)$ and set $\tau:=\lambda/(1-\lambda)$. Set $n_0:=-1$. Choose sequences $(n_j)_{j\in\NN}$ and $(l_j)_{j\in\NN}$ of positive integers such that $n_{j-1}+1<n_j-l_j$ for each $j$,
\begin{equation}
\lim_{j\to\infty}\frac{l_j}{n_j}=\lambda, 
\label{eq:L-proportion}
\end{equation}
and
\begin{equation}
\lim_{j\to\infty}\frac{n_{j-1}}{n_j}=0. 
\label{eq:fast-growth}
\end{equation}
These two assumptions together imply
\begin{equation}
\frac{1}{l_j}\sum_{i=1}^{j-1}l_i \to 0 \qquad\mbox{as $j\to\infty$},
\label{eq:l-growth-property}
\end{equation}
a fact we will use repeatedly. We partition $\NN$ into the four sets
\begin{gather*}
J_1:=\NN\cap \bigcup_{j=1}^\infty [n_{j-1}+2,n_j-l_j-1],\\
J_2:=\{n_j-l_j:j\in\NN\},\\
J_3:=\NN\cap \bigcup_{j=1}^\infty [n_j-l_j+1,n_j],\\
J_4:=\{n_j+1: j\in\NN\}.
\end{gather*}
Let $E_\lambda$ be the set of all sequences $(k_i)$ in $\{1,2,\dots,r\}^\NN$ for which $k_i=k^*$ if $i\in J_2\cup J_4$, and $k_i=r$ if $i\in J_3$. For a given sequence $(k_i)\in E_\lambda$, let
\[
L_n^+:=\max\{j\in\NN: k_{n-j+1}=k_{n-j+2}=\dots=k_n=r\},
\]
or $L_n^+=0$ if $k_n<r$; and 
\[
L_n^-:=\max\{j\in\NN: k_{n-j+1}=k_{n-j+2}=\dots=k_n=1\},
\]
or $L_n^-=0$ if $k_n>1$. We furthermore define
\[
s_{n,k}':=\#\{i\leq n: i\in J_1\ \mbox{and}\ k_i=k\}, \qquad k=1,2,\dots,r
\]
and $N_n:=\#\{i\leq n: i\in J_1\}$. For $n\in J_1$ and $j\in\NN$, let $\nu(n,j)$ be the index $\nu\in\NN$ such that $\#\{i\in J_1: \nu<i\leq n\}=j$ (or $\nu(n,j)=0$ if no such $\nu$ exists), and define
\[
\tilde{L}_n^+:=\max\{j\geq 1: k_i=r\ \mbox{for}\ \nu(n,j)<i\leq n, i\in J_1\},
\]
or $\tilde{L}_n^+=0$ if $k_n<r$. Thus, $\tilde{L}_n^+$ is the run length of the digit $r$ ending at the $n$th digit if digits whose index falls outside $J_1$ are ignored. Similarly, define 
\[
\tilde{L}_n^-:=\max\{j\geq 1: k_i=1\ \mbox{for}\ \nu(n,j)<i\leq n, i\in J_1\},
\]
or $\tilde{L}_n^+=0$ if $k_n>1$.

Next, for a probability vector $\mathbf{p}=(p_1,\dots,p_r)\in \Delta_r^0$, we let $\mu_{\mathbf{p},\lambda}$ denote the corresponding Bernoulli measure on $E_\lambda$. That is, the unique Borel measure on $E_\lambda$ (endowed with the product topology) under which $\{k_i: i\in J_1\}$ are independent random variables such that $k_i=k$ with probability $p_k$, for $k=1,\dots,r$. 

Fix $\alpha>1$, and assume $\mathbf{p}$ satisfies
\begin{equation}
\sum p_k(\log|d_k|-\alpha\log a_k)=\tau(\alpha-1)\log a_r.
\label{eq:p-tau-constraint}
\end{equation}
Let $E_{\mathbf{p},\lambda}$ be the set of sequences $(k_i)\in E_\lambda$ satisfying
\begin{gather}
\frac{s_{n,k}'}{N_n}\to p_k, \quad k=1,2,\dots,r, \label{eq:strong-law}\\
\frac{\tilde{L}_n^+}{N_n}\to 0, \label{eq:short-runs-r}\\
\frac{\tilde{L}_n^-}{N_n}\to 0. \label{eq:short-runs-1}
\end{gather}
Then $\mu_{\mathbf{p},\lambda}(E_{\mathbf{p},\lambda})=1$ by the strong law of large numbers and the well known fact that almost surely, maximum run lengths of digits grow at most logarithmically. Observe also that for $(k_i)\in E_{\mathbf{p},\lambda}$, \eqref{eq:short-runs-r} and \eqref{eq:short-runs-1} imply
\begin{equation}
\lim_{n\to\infty} \frac{L_n^-}{n}=0 \qquad\mbox{and}\qquad \lim_{n\to\infty, n\in J_1} \frac{L_n^+}{n}=0
\label{eq:L-limits}
\end{equation}
by the construction of $E_\lambda$, since $L_n^-\leq\tilde{L}_n^-+1$ for all $n\in\NN$, and $L_n^+\leq\tilde{L}_n^+$ for all $n\in J_1$.



Let $\pi:\{1,2,\dots,r\}\to [0,1]$ be the projection map given by
\[
\pi(k_1,k_2,\dots):=\lim_{n\to\infty} S_{k_1}\circ \dots \circ S_{k_n}(0),
\]
and let $\tilde{\mu}_{\mathbf{p},\lambda}:=\mu_{\mathbf{p},\lambda}\circ \pi^{-1}$ be the pushforward of $\mu_{\mathbf{p},\lambda}$. 

\begin{lemma}
We have
\[
\pi(E_{\mathbf{p},\lambda})\subset \{\xi\in(0,1)\backslash \mathcal{T}: \alpha_\phi^+(\xi)=\alpha\ \mbox{and}\ \alpha_\phi^-(\xi)\geq \alpha\} \subset E_\phi(\alpha).
\]
\end{lemma}

\begin{proof}
Since $|d_1|<a_1$ and $|d_r|<a_r$, it follows that $K_2\geq \max\{K_1,0\}$. Let $(k_i)\in E_{\mathbf{p},\lambda}$ and $\xi:=\pi((k_i))$. Clearly, $\xi\not\in\mathcal{T}$. By Theorem \ref{thm:exact-Holder-exponent},
\[
\alpha_\phi^+(\xi)=\gamma_2=\liminf_{n\to\infty}\frac{\sum_{k=1}^r s_{n,k}\log|d_k|+K_2\zeta_n L_n^+}{\sum_{k=1}^r s_{n,k}\log a_k}.
\]
Since $L_n^+/n\to 0$ along $n\in J_1$, this $\liminf$ is attained along $n=n_j$. Suppose $n=n_j$; then $L_n^+=l_j$ and so $k_{n-L_n^+}=k_{n_j-l_j}=k^*\in\Lambda$. Hence $\zeta_n=1$. Observe also that
\begin{align*}
\sum_{k=1}^r s_{n,k}\log|d_k| &= \sum_{k=1}^r s_{n,k}'\log|d_k|+\sum_{i=1}^j l_i \log|d_r|+(2j-1)\log|d_{k^*}|\\
&=\sum_{k=1}^r s_{n,k}'\log|d_k|+\sum_{i=1}^j l_i \log|d_r|+O(j)\\
&=\sum_{k=1}^r s_{n,k}'\log|d_k|+l_j \log|d_r|+o(l_j),
\end{align*}
and similarly,
\[
\sum_{k=1}^r s_{n,k}\log a_k=\sum_{k=1}^r s_{n,k}'\log a_k+l_j\log a_r+o(l_j),
\]
using \eqref{eq:l-growth-property}, which also implies that $j/l_j\to 0$ as $j\to\infty$. As a result,
\begin{align*}
\frac{\sum_{k=1}^r s_{n,k}\log|d_k|+K_2\zeta_n L_n^+}{\sum_{k=1}^r s_{n,k}\log a_k} 
&=\frac{\sum_{k=1}^r s_{n,k}'\log|d_k|+l_j\log a_r+o(l_j)}{\sum_{k=1}^r s_{n,k}'\log a_k+l_j\log a_r+o(l_j)}\\
&\to\frac{\sum_{k=1}^r p_k\log|d_k|+\tau\log a_r}{\sum_{k=1}^r p_k\log a_k+\tau\log a_r}=\alpha,
\end{align*}
where the first equality uses the definition of $K_2$, the last equality follows from \eqref{eq:p-tau-constraint}, and the convergence follows from \eqref{eq:strong-law} after dividing numerator and denominator by $N_{n_j}$, noting that $N_{n_j}=n_j-\sum_{i=1}^j l_i-(2j-1)\sim n_j-l_j$ by \eqref{eq:l-growth-property}, so that $l_j/N_{n_j}\sim l_j/(n_j-l_j)\to\lambda/(1-\lambda)=\tau$. Thus, $\alpha_\phi^+(\xi)=\alpha$.


Likewise, by \eqref{eq:short-runs-1} and the analog of Theorem \ref{thm:exact-Holder-exponent} for $\alpha_\phi^-$ (see Remark \ref{rem:left-Holder-exponent}) we have $\alpha_\phi^-(\xi)=\gamma_0(\xi)\geq \gamma_2(\xi)$. These two observations yield the first inclusion of the lemma; the second follows from Corollary \ref{cor:minimum}.
\end{proof}

We shall need the following lemma (see \cite[Proposition 4.9]{Falconer}), in which $B(x,\rho)$ denotes the open ball centered at $x$ with radius $\rho$, and $\mathcal{H}^s$ denotes $s$-dimensional Hausdorff measure.

\begin{lemma} \label{lem:Falconer}
Let $\mu$ be a mass distribution on $\RR^n$, let $F\subset \RR^n$ be a Borel set and let $0<c<\infty$ be a constant. If 
\[
\limsup_{\rho\to 0} \frac{\mu(B(x,\rho))}{\rho^s}<c \qquad\mbox{for all $x\in F$},
\]
then $\mathcal{H}^s(F)\geq \mu(F)/c$.
\end{lemma}

\begin{lemma} \label{lem:mass-upper-bound}
Let $\mathbf{p}\in\Delta_r^0$ satisfy \eqref{eq:p-tau-constraint}, and put
\[
s(\mathbf{p}):=\frac{(\alpha-1)\sum p_k\log p_k}{\sum p_k(\log|d_k|-\log a_k)}.
\]
Then for any $\xi\in\pi(E_{\mathbf{p},\lambda})$ and any $\eps>0$,
\[
\limsup_{\rho\to 0}\frac{\tilde{\mu}_{\mathbf{p},\lambda}(B(\xi,\rho))}{\rho^{s(\mathbf{p})-\eps}}=0.
\]
\end{lemma}

\begin{proof}
Fix $\xi\in\pi(E_{\mathbf{p},\lambda})$ and $0<\eps<s(\mathbf{p})$. We first show that
\begin{equation}
\limsup_{n\to\infty} \frac{1}{N_n}\big(\log \tilde{\mu}_{\mathbf{p},\lambda}(I_n(\xi))-(s(\mathbf{p})-\eps)\log|I_n(\xi)|\big)<0.
\label{eq:negative-limsup}
\end{equation}
Note that subject to \eqref{eq:p-tau-constraint}, $s(\mathbf{p})$ can be written as
\begin{equation}
s(\mathbf{p}):=\frac{\sum p_k\log p_k}{\sum p_k\log a_k+\tau\log a_r}.
\label{eq:alternative-s-expression}
\end{equation}
It suffices to consider $n\in J_1\cup J_3$. For all $n$ we have $\tilde{\mu}_{\mathbf{p},\lambda}(I_n(\xi))=\prod_{k=1}^r p_k^{s_{n,k}'}$, and hence
\begin{equation}
\lim_{n\to\infty} \frac{\log \tilde{\mu}_{\mathbf{p},\lambda}(I_n(\xi))}{N_n}=\sum_{k=1}^r p_k\log p_k
\label{eq:measure-limit}
\end{equation}
by \eqref{eq:strong-law}. Suppose first that $n_j-l_j<n\leq n_j$. Then $N_n=n_j-\sum_{i=1}^{j}l_i-(2j-1)=n_j-l_j+o(l_j)$  by \eqref{eq:l-growth-property}, and
\begin{align*}
\log|I_n(\xi)|=\sum_{k=1}^r s_{n,k}\log a_k
&=\sum_{k=1}^r s_{n,k}'\log a_k+\left(\sum_{i=1}^{j}l_i+n-n_j\right)\log a_r+O(j)\\
&\geq \sum_{k=1}^r s_{n,k}'\log a_k+l_j\log a_r+o(l_j).
\end{align*}
Since $l_j/N_n\sim l_j/(n_j-l_j)\to \lambda/(1-\lambda)$ by \eqref{eq:L-proportion}, it follows that
\begin{equation}
\liminf_{n\to\infty,n\in J_3} \frac{\log|I_n(\xi)|}{N_n}\geq 
\sum_{k=1}^r p_k\log a_k+\frac{\lambda}{1-\lambda}\log a_r=\sum_{k=1}^r p_k\log a_k+\tau\log a_r.
\label{eq:J3-diameter-limit}
\end{equation}

Assume next that $n_{j-1}<n<n_j-l_j$. Then by \eqref{eq:l-growth-property},
\[
N_n=n-\sum_{i=1}^{j-1}l_i-2(j-1)>n_{j-1}-\sum_{i=1}^{j-1}l_i-2(j-1)=n_{j-1}-l_{j-1}+o(l_{j-1}),
\] 
and
\begin{align*}
\log|I_n(\xi)|&=\sum_{k=1}^r s_{n,k}'\log a_k+\left(\sum_{i=1}^{j-1}l_i\right)\log a_r+O(j)\\
&=\sum_{k=1}^r s_{n,k}'\log a_k+l_{j-1}\log a_r+o(l_{j-1}).
\end{align*}
Hence, 
\begin{align}
\begin{split}
\liminf_{n\to\infty,n\in J_1} \frac{\log|I_n(\xi)|}{N_n}&\geq \sum_{k=1}^r p_k\log a_k+\left(\lim_{j\to\infty}\frac{l_{j-1}}{n_{j-1}-l_{j-1}}\right)\log a_r\\
&=\sum_{k=1}^r p_k\log a_k+\tau\log a_r,
\end{split}
\label{eq:J1-diameter-limit}
\end{align}
again using \eqref{eq:L-proportion}. The derivation for $n\in J_2$ and $n\in J_4$ is essentially the same. Combining \eqref{eq:alternative-s-expression}-\eqref{eq:J1-diameter-limit} yields \eqref{eq:negative-limsup}. Since $N_n\to\infty$, it follows that
\begin{equation}
\limsup_{n\to\infty} \frac{\tilde{\mu}_{\mathbf{p},\lambda}(I_n(\xi))}{|I_n(\xi)|^{s(\mathbf{p})-\eps}}=0.
\label{eq:zero-limit-for-cylinders}
\end{equation}

Now consider an arbitrary open ball $B(\xi,\rho)=(\xi-\rho,\xi+\rho)$. Let $n$ be the largest integer such that
\begin{equation}
\rho\leq a_{\min}|I_n(\xi)|=a_{\min} a_{k_1}\dots a_{k_n}.
\label{eq:choice-of-n-2}
\end{equation} 
Then $\rho\asymp |I_n(\xi)|=a_{k_1}\dots a_{k_n}$. If $\xi+\rho\in I_n(\xi)$, then
\begin{equation}
\frac{\tilde{\mu}_{\mathbf{p},\lambda}([\xi,\xi+\rho))}{\rho^{s(\mathbf{p})-\eps}}\to 0
\label{eq:zero-limit-half-ball}
\end{equation}
by \eqref{eq:zero-limit-for-cylinders}. So suppose $\xi+\rho\not\in I_n(\xi)$. Let $l:=L_n^+(\xi)$. Note that $|I_n(\xi)|=a_{k_1}\cdots a_{k_{n-l}}a_r^l$. Now let $m$ be the largest integer such that
\begin{equation}
a_{k_1}\cdots a_{k_{n-l-1}}a_{k_{n-l}+1}a_1^{l+m}\geq \rho.
\label{eq:choice-of-m}
\end{equation}
Observe that $l+m\geq 0$ in view of \eqref{eq:choice-of-n-2}. Hence, there is an index $j'$ such that $|I_{n+m,j'}|=a_{k_1}\cdots a_{k_{n-l-1}}a_{k_{n-l}+1}a_1^{l+m}$, the interval $I_{n+m,j'}$ is adjacent to $I_n(\xi)$, and by \eqref{eq:choice-of-m}, $\xi+\rho\in I_{n+m,j'}$. If $n\not\in J_1$, then the coding $(k_1',k_2',\dots k_{n+m}')$ of $I_{n+m,j'}$ is not (a prefix of a sequence) in $E_\lambda$, so $\tilde{\mu}_{\mathbf{p},\lambda}(I_{n+m,j'})=0$. Consider $n\in J_1$. Then, assuming $n-l\in J_1$ as well,
\begin{equation}
\tilde{\mu}_{\mathbf{p},\lambda}(I_{n+m,j'})=\left(\prod_{k=1}^r p_k^{s_{n-l-1,k}'}\right)p_{k_{n-l}+1}p_1^{l+m}=\tilde{\mu}_{\mathbf{p},\lambda}(I_n(\xi))\cdot\frac{p_{k_{n-l}+1}}{p_{k_{n-l}}}\cdot\frac{p_1^{l+m}}{p_r^l}.
\label{eq:measure-expansion}
\end{equation}
(If $n-l\not\in J_1$, then $n-l\in J_4$ and hence $\tilde{\mu}_{\mathbf{p},\lambda}(I_{n+m,j'})=0$.)
Note that by the analog of \eqref{eq:d-power-asymp} we have
\begin{equation}
\frac{p_1^{l+m}}{p_r^l} \asymp \left(\frac{p_1^{\log a_r/\log a_1}}{p_r}\right)^l.
\label{eq:p-asymptotics}
\end{equation}
Moreover, as a consequence of \eqref{eq:l-growth-property}, $\liminf_{n\to\infty} N_n/n=1-\lambda>0$.
Using \eqref{eq:L-limits}, it follows that $L_n^+/N_n\to 0$, and hence (recalling that $l=L_n^+$), \eqref{eq:measure-expansion} and \eqref{eq:p-asymptotics} imply
\[
\limsup_{n\to\infty,n\in J_1} \frac{\log\tilde{\mu}_{\mathbf{p},\lambda}(I_{n+m,j'})}{N_n}=\limsup_{n\to\infty,n\in J_1} \frac{\log\tilde{\mu}_{\mathbf{p},\lambda}(I_n(\xi))}{N_n}.
\]
Now \eqref{eq:zero-limit-half-ball} follows as in the proof of \eqref{eq:zero-limit-for-cylinders}, since $[\xi,\xi+\rho)\subset I_n(\xi)\cup I_{n+m,j'}$. In the same way (considering $L_n^-$ instead of $L_n^+$), we can show that
\[
\frac{\tilde{\mu}_{\mathbf{p},\lambda}((\xi-\rho,\xi))}{\rho^{s(\mathbf{p})-\eps}}\to 0.
\]
Thus, the proof is complete.
\end{proof}

\begin{proof}[Proof of the lower bound in Theorem \ref{thm:main}]

(a) Assume first the hypotheses of Theorem \ref{thm:main}\,({\em a}). We need to show that $D(\alpha)\geq \beta^*(\alpha)$. But this follows easily from Proposition \ref{prop:duality}, since $E_\phi(\alpha)$ contains the set
\[
F_{\mathbf p}:=\{\xi\in(0,1): s_{n,k}(\xi)/n\to p_k, \quad k=1,2,\dots,r\}
\]
whenever $\mathbf{p}\in\Delta_r^0$ with $\sum_{k\in\II_+} p_k(\log|d_k|-\alpha\log a_k)=0$, and $\dim_H F_{\mathbf p}=H(p_1,\dots,p_r)$, a generalization of Eggleston's theorem \cite{Eggleston} due to Li and Dekking \cite{Li-Dekking}. After all, the existence of the frequencies $\lim_{n\to\infty}s_{n,k}/n$ means that $L_n^+/n\to 0$, so 
\[
\alpha_\phi^+(\xi)=\gamma_0(\xi)=\frac{\sum_{k\in\II_+}p_k\log|d_k|}{\sum_{k\in\II_+}p_k\log a_k}=\alpha,
\] 
and in the same way, $\alpha_\phi^-(\xi)=\gamma_0(\xi)$.

(b) Assume next the hypotheses of Theorem \ref{thm:main}\,({\em b}). Fix $\alpha>1$ (when $\alpha=1$ there is nothing to prove). By Proposition \ref{prop:duality-b} and by continuity, it suffices to show that $D(\alpha)\geq s(\mathbf{p})$ whenever $\mathbf{p}\in\Delta_r^0$ and $\sum p_k(\log|d_k|-\alpha\log a_k)<0$, where $s(\mathbf{p})$ was defined in Lemma \ref{lem:mass-upper-bound}. But when this last inequality holds, there is a number $\tau>0$ such that \eqref{eq:p-tau-constraint} holds, and we can find $\lambda\in(0,1)$ such that $\lambda/(1-\lambda)=\tau$. It thus follows from the three preceding Lemmas that
\[
\dim_H E_\phi(\alpha)\geq \dim_H \pi(E_{\mathbf{p},\lambda})\geq s(\mathbf{p}),
\]
as desired.
\end{proof}

\section{Proof of the upper bound} \label{sec:upper-bound}

In this section we prove the upper bounds in Theorem \ref{thm:main}\,({\em a}) and ({\em b}). For a short proof of the following elementary lemma, see \cite{Allaart}.

\begin{lemma} \label{lem:multinomial-estimate}
Let $n\in\NN$, and let $m_1,\dots,m_r$ be nonnegative integers with $\sum_{k=1}^r m_k=n$. Put $p_k:=m_k/n$ for $k=1,\dots,r$. Then
\begin{equation*}
\frac{n!}{m_1!\cdots m_r!}\leq 2\sqrt{n}\left(\prod_{k=1}^r p_k^{-p_k}\right)^n,
\end{equation*}
where we use the convention $0^0\equiv 1$.
\end{lemma}

The next result is probably known. However, since the author could not find it in the literature, a proof is included for completeness.

\begin{lemma} \label{lem:extreme-run-length}
The set 
\[
R_1:=\left\{\xi\in(0,1): \limsup_{n\to\infty}\frac{L_n^+(\xi)}{n}=1\right\}
\]
has Hausdorff dimension zero.
\end{lemma}

\begin{proof}
Fix $s>0$ and let 
\[
M:=\max\left\{\prod_{k=1}^r p_k^{-p_k}a_k^{p_k s}: \mathbf{p}=(p_1,\dots,p_r)\in\Delta_r\right\}.
\] 
Choose $\eps>0$ small enough so that
\begin{equation}
M_\eps:=M^\eps a_r^{(1-\eps)s}<1.
\label{eq:less-than-one}
\end{equation}
This is possible since $a_r^s<1$. If $\xi\in R_1$, then for every $N\in\NN$ there is an $n\geq N$ such that $L_n^+(\xi)>(1-\eps)n$. For such $n$, and for any $r$-tuple of integers $(m_1,\dots,m_r)$ with $\sum_{k=1}^r m_k=\lceil \eps n\rceil$, there are $\binom{\lceil \eps n\rceil}{m_1,\dots,m_r}$ basic intervals of level $n$ satisfying $s_{n-\lceil \eps n\rceil,k}=m_k$, $k=1,\dots,r$, that could possibly contain $\xi$. Each of these intervals has length $\big(\prod_{k=1}^r a_k^{m_k}\big)a_r^{n-\lceil \eps n\rceil}$. Given $\eta>0$, choose $N$ so large that $a_{\max}^N<\eta$. Setting $p_k:=m_k/\lceil \eps n\rceil$, we then have
\begin{align*}
\mathcal{H}_\eta^s(R_1) &\leq \sum_{n=N}^\infty \sum_{m_1+\dots+m_r=\lceil \eps n\rceil} \binom{\lceil \eps n\rceil}{m_1,\dots,m_r} \left(\prod_{k=1}^r a_k^{m_k s}\right)a_r^{(n-\lceil \eps n\rceil)s}\\
&\leq 2\sum_{n=N}^\infty (\lceil \eps n\rceil+1)^r \sqrt{\lceil \eps n\rceil} \max_{\mathbf{p}\in\Delta_r}\left\{\left(\prod_{k=1}^r p_k^{-p_k}\right)^{\lceil \eps n\rceil}\left(\prod_{k=1}^r a_k^{p_k}\right)^{\lceil \eps n\rceil s}\right\}a_r^{(n-\lceil \eps n\rceil)s}\\
&\leq 2\sum_{n=N}^\infty (\eps n+2)^{r+1} \max_{\mathbf{p}\in\Delta_r}\left\{\left[\left(\prod_{k=1}^r p_k^{-p_k}a_k^{p_k s}\right)^\eps a_r^{(1-\eps)s}\right]^n \left(\prod_{k=1}^r p_k^{-p_k}\right)\right\}a_r^{-s}\\
&\leq 2\sum_{n=N}^\infty (\eps n+2)^{r+1} \max_{\mathbf{p}\in\Delta_r}\left(\prod_{k=1}^r p_k^{-p_k}\right)a_r^{-s}M_\eps^n\\
&\to 0 \qquad\mbox{as $N\to\infty$},
\end{align*}
where the second inequality follows from Lemma \ref{lem:multinomial-estimate}, and the last line follows from \eqref{eq:less-than-one}.
Hence, $\mathcal{H}^s(R_1)=0$, and since $s>0$ was arbitrary, the lemma follows.
\end{proof}

\begin{remark} \label{rem:large-l-small-dimension}
{\rm
Essentially the same argument shows this: For any $s>0$, there exists $\eps>0$ such that
\begin{equation}
\dim_H\left(\left\{\xi\in(0,1): \limsup_{n\to\infty}\frac{L_n^+(\xi)}{n}>1-\eps\right\}\right)<s.
\label{eq:large-l-small-dimension}
\end{equation}
}
\end{remark}

\begin{proof}[Proof of the upper bound in Theorem \ref{thm:main}]
Let $E_\phi^+(\alpha):=\{\xi: \alpha_\phi^+(\xi)=\alpha\}$ and likewise, $E_\phi^-(\alpha):=\{\xi: \alpha_\phi^-(\xi)=\alpha\}$. 
Since $\alpha_\phi(\xi)=\min\{\alpha_\phi^+(\xi),\alpha_\phi^-(\xi)\}$ for all but countably many $\xi$ by Corollary \ref{cor:minimum}, it follows that
\[
D(\alpha)=\dim_H E_\phi(\alpha)\leq \max\left\{\dim_H E_\phi^+(\alpha),\dim_H E_\phi^-(\alpha)\right\}.
\]
Thus, it suffices to estimate $\dim_H E_\phi^+(\alpha)$ and $\dim_H E_\phi^-(\alpha)$ separately. We focus on $\dim_H E_\phi^+(\alpha)$; the other term is dealt with similarly.

We can write $E_\phi^+(\alpha)=E_1(\alpha)\cup E_2(\alpha)$, where
\begin{align*}
E_1(\alpha):&=\{\xi: \alpha_\phi^+(\xi)=\min\{\gamma_0(\xi),\gamma_1(\xi)\}=\alpha\},\\
E_2(\alpha):&=\{\xi: \alpha_\phi^+(\xi)=\gamma_2(\xi)=\alpha<\min\{\gamma_0(\xi),\gamma_1(\xi)\}\}.
\end{align*}
It is proved in \cite[Section 7]{Allaart} that 
\[
\dim_H \{\xi:\min\{\gamma_0(\xi),\gamma_1(\xi)\}=\alpha\}\leq \beta^*(\alpha),
\]
so it remains to bound the dimension of $E_2(\alpha)$.
If $\alpha<1$, then $E_2(\alpha)=\emptyset$ by Lemma \ref{lem:gammas}, so we need only consider the case $\alpha\geq 1$. First, for $\alpha=1$, $E_2(\alpha)$ contains only points $\xi$ with $\limsup L_n^+(\xi)/n=1$ by the second statement of Lemma \ref{lem:gammas}, so $\dim_H E_2(1)=0$ by Lemma \ref{lem:extreme-run-length}. Assume then, for the remainder of the proof, that $\alpha>1$.

Let $h(\alpha)$ be defined as in Proposition \ref{prop:duality-b}:
\[
h(\alpha):=\max\left\{\frac{(\alpha-1)\sum p_k\log p_k}{\sum p_k(\log|d_k|-\log a_k)}: \mathbf{p}\in\Delta_r^0,\ \sum p_k(\log|d_k|-\alpha\log a_k)\leq 0\right\}.
\]
Note that $h(\alpha)$ is well defined even when $|d_k|\geq a_k$ for some $k$, since the constraint $\sum p_k(\log|d_k|-\alpha\log a_k)\leq 0$ implies 
\begin{equation}
\sum p_k(\log|d_k|-\log a_k)\leq (\alpha-1)\sum p_k\log a_k\leq (\alpha-1)\log a_{\max}<0,
\label{eq:log-difference-estimate}
\end{equation}
so the denominator is bounded away from zero. We will show that $\dim_H E_2(\alpha)\leq h(\alpha)$. 

Fix $s>h(\alpha)$. We can choose $\delta>0$ so small that $\alpha-\delta>1$ and $s>s_\delta$, where
\begin{align}
\begin{split}
s_\delta&:=\sup\left\{\frac{(\alpha+\delta-1)\sum p_k\log p_k}{\sum p_k(\log|d_k|-\log a_k)}: \mathbf{p}\in\Delta_r^0\right.\\
& \left.\phantom{\frac{(\alpha+\delta-1)\sum p_k\log p_k}{\sum p_k(\log|d_k|-\log a_k)}}\mbox{and}\quad \sum p_k(\log|d_k|-(\alpha-\delta)\log a_k)\leq 0\right\}.
\end{split}
\label{eq:delta-constrained-max}
\end{align}
In view of Remark \ref{rem:large-l-small-dimension}, we may choose $\eps>0$ so that \eqref{eq:large-l-small-dimension} holds with $2\eps$ in place of $\eps$. Hence, it is sufficient to show that $\dim_H E_2^\eps(\alpha)\leq s$, where
\[
E_2^\eps(\alpha):=\left\{\xi\in E_2(\alpha): \limsup_{n\to\infty} \frac{L_n^+(\xi)}{n}\leq 1-2\eps\right\}.
\]
Take a point $\xi\in E_2^\eps(\alpha)$. Write $l_n:=L_n^+(\xi)$ and $s_{n,k}:=s_{n,k}(\xi)$. Since $\gamma_2(\xi)=\alpha<\gamma_0(\xi)$, it follows that
\[
\liminf_{n\to\infty} \frac{\sum s_{n,k}\log|d_k|+K_2 l_n}{\sum s_{n,k}\log a_k}=\liminf_{n\to\infty} \frac{\sum s_{n-l_n,k}\log|d_k|+l_n\log a_r}{\sum s_{n-l_n,k}\log a_k+l_n\log a_r}=\alpha.
\]
So on the one hand, there is an integer $N$ such that
\begin{equation}
\sum s_{n-l_n,k}\log|d_k|+l_n\log a_r<(\alpha-\delta)\left(\sum s_{n-l_n,k}\log a_k+l_n\log a_r\right)
\label{eq:delta-upper}
\end{equation}
for all $n\geq N$; and on the other hand,
\begin{equation}
\sum s_{n-l_n,k}\log|d_k|+l_n\log a_r>(\alpha+\delta)\left(\sum s_{n-l_n,k}\log a_k+l_n\log a_r\right)
\label{eq:delta-lower}
\end{equation}
for infinitely many $n$. So for all $N\in\NN$, we can find $n\geq N$ such that \eqref{eq:delta-upper} and \eqref{eq:delta-lower} hold. Furthermore, $n$ can be chosen large enough so that $l_n<(1-\eps )n$.
Fix such an $n$, and set $l:=l_n$ and
\[
p_k:=\frac{s_{n-l,k}}{n-l}, \qquad k=1,\dots,r.
\]
Then $\mathbf{p}:=(p_1,\dots,p_r)\in\Delta_r^0$ (for else $\alpha_\phi(\xi)$ would equal $+\infty$), and \eqref{eq:delta-lower} can be rearranged to obtain
\begin{equation}
\sum p_k\big(\log|d_k|-(\alpha+\delta)\log a_k\big)>(\alpha+\delta-1)\frac{l}{n-l}\log a_r,
\label{eq:p-lower}
\end{equation}
while \eqref{eq:delta-upper} gives
\begin{equation}
\sum p_k\big(\log|d_k|-(\alpha-\delta)\log a_k\big)<(\alpha-\delta-1)\frac{l}{n-l}\log a_r\leq 0,
\label{eq:p-upper}
\end{equation}
since $\alpha-\delta-1>0$ and $\log a_r<0$. Note in particular that $\mathbf{p}$ satisfies the constraint in \eqref{eq:delta-constrained-max}. We can now calculate, using \eqref{eq:p-lower}, \eqref{eq:delta-constrained-max}, \eqref{eq:p-upper} and the obvious modification of \eqref{eq:log-difference-estimate}:
\begin{align*}
-\sum p_k\log p_k &+ s\left(\sum p_k\log a_k+\frac{l}{n-l}\log a_r\right)\\
&\leq -\sum p_k\log p_k+s(\alpha+\delta-1)^{-1}\sum p_k(\log|d_k|-\log a_k)\\
&\leq (s-s_\delta)(\alpha+\delta-1)^{-1}\sum p_k(\log|d_k|-\log a_k)\\
&\leq (s-s_\delta)(\alpha+\delta-1)^{-1}(\alpha-\delta-1)\log a_{\max}\\
&<0.
\end{align*}
Hence, there is a constant $C_\delta<1$ (not depending on $n$) such that
\begin{equation}
\left(\prod p_k^{-p_k}a_k^{p_k s}\right)a_r^{\left(\frac{l}{n-l}\right)s}\leq C_\delta.
\label{eq:away-from-one}
\end{equation}
Note that $\xi$ must lie in one of $\binom{n-l}{s_{n-l,1},\dots,s_{n-l,r}}$ basic intervals of level $n$, each of which has length
\[
\prod a_k^{s_{n,k}}=\left(\prod a_k^{s_{n-l,k}}\right)a_r^l=\left(\prod a_k^{p_k}\right)^{n-l}a_r^l.
\]
Given $\eta>0$, $N$ can be taken large enough so that $a_{\max}^N<\eta$, so all basic intervals of level $N$ or greater have diameter less than $\eta$. Let $\mathcal{M}_{n,l,\delta}$ be the set of all $r$-tuples $(m_1,\dots,m_r)$ for which $\sum m_k=n-l$, $\sum m_k\big(\log|d_k|-(\alpha-\delta)\log a_k\big)\leq 0$, and $\sum m_k\big(\log|d_k|-(\alpha+\delta)\log a_k\big)\geq (\alpha+\delta-1)l\log a_r$. Let $\mathcal{P}_{n,l,\delta}$ denote the set of corresponding probability vectors, i.e. those $\mathbf{p}\in\Delta_r^0$ satisfying \eqref{eq:p-lower} and the constraint in \eqref{eq:delta-constrained-max}.
Then
\begin{align*}
\mathcal{H}_\eta^s(E_2^\eps(\alpha)) &\leq \sum_{n=N}^\infty \sum_{l=0}^{\lfloor (1-\eps)n\rfloor} \sum_{(m_1,\dots,m_r)\in\mathcal{M}_{n,l,\delta}} \binom{n-l}{m_1,\dots,m_r}\left[\left(\prod a_k^{m_k}\right)a_r^l\right]^s\\
&\leq \sum_{n=N}^\infty \sum_{l=0}^{\lfloor (1-\eps)n\rfloor} (n+1)^r \sup_{(m_1,\dots,m_r)\in\mathcal{M}_{n,l,\delta}}\binom{n-l}{m_1,\dots,m_r}\left[\left(\prod a_k^{m_k}\right)a_r^l\right]^s\\
&\leq \sum_{n=N}^\infty \sum_{l=0}^{\lfloor (1-\eps)n\rfloor} (n+1)^r \cdot 2\sqrt{n-l} \sup_{\mathbf{p}\in\mathcal{P}_{n,l,\delta}} \left[\left(\prod p_k^{-p_k}a_k^{p_k s}\right)a_r^{\left(\frac{l}{n-l}\right)s}\right]^{n-l}\\
&\leq 2\sum_{n=N}^\infty \sum_{l=0}^{\lfloor (1-\eps)n\rfloor} (n+1)^{r+1} C_\delta^{n-l}\\
&\leq 2\sum_{n=N}^\infty (n+1)^{r+1} \frac{C_\delta^{\eps n-1}}{1-C_\delta}.
\end{align*}
Here, the third inequality follows from Lemma \ref{lem:multinomial-estimate} and the fourth from \eqref{eq:away-from-one}.
The final series converges; thus, letting $\eta\searrow 0$ (and hence $N\to\infty$), we obtain $\mathcal{H}^s(E_2^\eps(\alpha))=0$. Therefore, $\dim_H(E_2^\eps(\alpha))\leq s$, as was to be shown.

The upper bound in Theorem \ref{thm:main}\,({\em b}) now follows, using Proposition \ref{prop:duality-b}. To see the upper bound in Theorem \ref{thm:main}\,({\em a}), make the following two observations: First, if $|d_k|\geq a_k$ for some $k$, then one can check easily that the constrained maximum in the definition of $h(\alpha)$ is achieved on the boundary $\sum p_k(\log|d_k|-\alpha\log a_k)=0$, from which it follows that $h(\alpha)\leq \beta^*(\alpha)$ via Proposition \ref{prop:duality}. Second, if $\Lambda=\emptyset$, then there are no points $\xi$ for which $\gamma_2(\xi)<\gamma_0(\xi)$, so $E_2(\alpha)=\emptyset$. In both cases, we conclude that $D(\alpha)\leq\beta^*(\alpha)$.
\end{proof}

\section{Proof of differentiability of $D(\alpha)$} \label{sec:smooth-connection}

Assume in this section that $|d_k|<a_k$ for every $k$.

\begin{proposition}
The function $D(\alpha)$ is differentiable at $\alpha_0$.
\end{proposition}

\begin{proof}
Recall that $\sigma$ is the unique positive number such that $\sum_{k=1}^r (|d_k|/a_k)^\sigma=1$ and that $p_k^*=(|d_k|/a_k)^\sigma$ for $k=1,\dots,r$. Since $\beta(q)$ solves \eqref{eq:scaling-equation}, it follows immediately that $\beta(\sigma)=-\sigma$.

Let $q^*(\alpha)$ satisfy $\beta^*(\alpha)=\alpha q^*(\alpha)+\beta(q^*(\alpha))$. Then $\beta'(q^*(\alpha))=-\alpha$. Differentiating \eqref{eq:scaling-equation} with respect to $q$ then gives
\[
\sum_{i\in\II^+}|d_k|^q a_k^{\beta(q)}\left(\log|d_k|-\alpha\log a_k\right)=0.
\]
When $\alpha=\alpha_0$, this equation is satisfied by $q=\sigma$. Since the infimum in \eqref{eq:legendre-transforms} is uniquely attained, it follows that $q^*(\alpha_0)=\sigma$, and hence, $\beta^*(\alpha_0)=\alpha_0\sigma+\beta(\sigma)=\sigma(\alpha_0-1)$. Thus, $D(\alpha)$ is continuous at $\alpha_0$.

Next, observe that
\[
{\beta^*}'(\alpha)={\alpha q^*}'(\alpha)+q^*(\alpha)+\beta'(q^*(\alpha)){q^*}'(\alpha)=q^*(\alpha),
\]
and so ${\beta^*}'(\alpha_0)=q^*(\alpha_0)=\sigma$. Therefore, $D(\alpha)$ is differentiable at $\alpha_0$.
\end{proof}

We now briefly address the remaining statements of Theorem \ref{thm:main}. First, statements ({\em a})\,(i) and ({\em b})\,(i) follow from Corollary \ref{cor:alpha-range}, Corollary \ref{cor:minimum} and Remark \ref{rem:cut-points}. Statements ({\em a})\,(ii) and ({\em b})\,(ii) are clear. Statement ({\em a})\,(v) and its counterpart in ({\em b})\,(vi) are straightforward; see \cite{Allaart}. That $\alpha_0\leq\hat{\alpha}$ now follows from the differentiability of $D(\alpha)$, which implies ${\beta^*}'(\alpha_0)\geq 0$. Finally, statements ({\em a})\,(iv) and ({\em b})\,(v) are proved exactly as in \cite{Allaart}.

\section{Higher dimensional self-affine functions} \label{sec:higher-dim}

In \cite{Allaart}, the present author studied the pointwise H\"older spectrum of certain self-affine functions $f: [0,1]\to\RR^d$. In terms of the notation in the present article, they can be described as follows: Let $0=x_0<x_1<x_2<\dots<x_r=1$ be a partition of $[0,1]$, and let $(y_k)_{k=0}^r$ be points in $\RR^d$ such that $y_0=\mathbf{0}:=(0,0,\dots,0)$, $y_r=\mathbf{e_1}:=(1,0,\dots,0)$, and $|y_k-y_{k-1}|<1$ for $k=1,2,\dots,r$. For each $k\in\{1,\dots,r\}$, let $S_k(x)=a_k x+b_k$ and let $\Psi_k:\RR^d\to\RR^d$ be a contractive similarity such that the maps
\[
T_k(x,y):=(S_k(x),\Psi_k(y)),  \qquad k=1,2,\dots,r
\]
satisfy the connectivity conditions
\[
T_k(\{(0,\mathbf{0}),(1,\mathbf{e_1})\})=\{(x_{k-1},y_{k-1}),(x_k,y_k)\}, \qquad k=1,2,\dots,r.
\]
There is then a unique continuous function $f:[0,1]\to\RR^d$ such that
\begin{equation}
f(a_k x+b_k)=\Psi_k(f(x)), \qquad k=1,2,\dots,r, \quad x\in[0,1].
\label{eq:higher-dimensional-self-affine}
\end{equation}
Let $d_k:=|y_k-y_{k-1}|$, so $d_k$ is the Lipschitz constant (or contraction ratio) of $\Psi_k$. Note that in this setup, $d_k\geq 0$ for every $k$. Moreover, since the maps $T_k$ factor we have necessarily that $\sum_k d_k\geq 1=\sum |a_k|$. Thus, a situation as in Theorem \ref{thm:main}\,({\em b}) cannot occur for these functions.

In \cite{Allaart}, two types of H\"older exponent are considered: One is the pointwise H\"older exponent $\alpha_f(\xi)$ defined in the Introduction of the present article; the other is
\[
\tilde{\alpha}_f(\xi):=\sup\left\{\alpha\geq 0: \limsup_{x\to\xi}\frac{|f(x)-f(\xi)|}{|x-\xi|^\alpha}=0\right\}.
\]
Clearly, $\alpha_f(\xi)\geq \tilde{\alpha}_f(\xi)$, but {\em a priori} these exponents need not be equal. One of the main results of \cite{Allaart} is that
\begin{equation}
\dim_H \{\xi\in(0,1): \tilde{\alpha}_f(\xi)=\alpha\}=\beta^*(\alpha), \qquad \alpha\in(\alpha_{\min},\alpha_{\max}),
\label{eq:higher-dim-spectrum}
\end{equation}
with $\alpha_{\min}$, $\alpha_{\max}$ and $\beta^*(\alpha)$ defined as in Section \ref{sec:intro} (see \cite[Theorem 2.7]{Allaart}). As in the present paper, the proof involves an exact expression for $\tilde{\alpha}_f(\xi)$, which can be reformulated in terms of our notation from Section \ref{sec:exact-Holder-exponent} as
\begin{equation}
\tilde{\alpha}_f(\xi)=\gamma_1(\xi),
\label{eq:higher-dim-exponent}
\end{equation}
under the simplifying assumption that $a_k>0$ for each $k$ and the constant $K_1$ from Section \ref{sec:exact-Holder-exponent} is positive. (When $a_k<0$ for one or more $k$'s, the expression is more complicated; see \cite[Theorem 6.1]{Allaart}. When $K_1<0$, one interchanges the roles of the digits $1$ and $r$.) Using the method of divided differences that was employed in the proof of Theorem \ref{thm:exact-Holder-exponent} (based on Dubuc's lemma), it is straightforward to ``upgrade" the proof of \cite[Theorem 6.1]{Allaart} and show that \eqref{eq:higher-dim-exponent} (and hence \eqref{eq:higher-dim-spectrum}) holds with $\alpha_f(\xi)$ in place of $\tilde{\alpha}_f(\xi)$. Even in the case when $a_k<0$ for some $k$, this still works. 

In \cite[Theorem 2.5]{Allaart}, it is shown that $\alpha_f(\xi)=\tilde{\alpha}_f(\xi)$ for the special case when $a_1=\dots=a_r=1/r$. We can now conclude that $\alpha_f(\xi)=\tilde{\alpha}_f(\xi)$ for {\em any} function $f$ of the form \eqref{eq:higher-dimensional-self-affine}. (This fails in general for the self-affine functions $\phi$ from \eqref{eq:functional-equation}, which can have a nonzero finite derivative as shown in Theorem \ref{thm:exact-Holder-exponent}.)

\section*{Acknowledgments}
This paper could not have been written without the prior publication of Dubuc's article \cite{Dubuc}. I wish to thank Prof.~Serge Dubuc for encouraging me to use his methods (especially, the divided difference technique and Lemma \ref{lem:oscillation-bound}) in order to extend his results.



\begin{thebibliography}{23}

\bibitem{Allaart}
{\sc P. C. Allaart}, Differentiability and H\"older spectra of a class of self-affine functions. {\em Adv. Math.} {\bf 328} (2018), 1--39.




\bibitem{BKK}
{\sc B. B\'ar\'any, G. Kiss} and {\sc I. Kolossv\'ary}, Pointwise regularity of parameterized affine zipper fractal curves. {\em Nonlinearity} {\bf 31} (2018), no. 5, 1705--1733.


\bibitem{Bedford}
{\sc T. Bedford}, H\"older exponents and box dimension for self-affine fractal functions. {\em Constr. Approx.} {\bf 5} (1989), 33--48.

\bibitem{BenSlimane}
{\sc M. Ben Slimane}, Multifractal formalism for selfsimilar functions expanded in singular basis. {\em Appl. Comput. Harmon. Anal.} {\bf 11} (2001), 387--419.





\bibitem{Dubuc}
{\sc S. Dubuc}, Non-differentiability and H\"older properties of self-affine functions. {\em Expo. Math.} {\bf 36} (2018), 119--142.

\bibitem{Eggleston}
{\sc H. Eggleston}, The fractional dimension of a set defined by decimal properties. {\em Quart. J. Math. Oxford Ser.} {\bf 20} (1949), 31--36.

\bibitem{Falconer}
{\sc K.~J.~Falconer}, {\it Fractal Geometry. Mathematical Foundations and Applications}, 2nd Edition, Wiley (2003)

\bibitem{Frisch}
{\sc U. Frisch} and {\sc G. Parisi}, Fully developed turbulence and intermittency, in Proc. Enrico Fermi, International Summer School in Physics, North Holland, Amsterdam (1985), 84-88.

\bibitem{Hellinger}
{\sc E. Hellinger}, Die Orthogonalinvariante Quadratischer Formen von Unendlichvielen Variabelen. Dissertation, G\"ottingen, 1907.

\bibitem{JS}
{\sc J. Jaerisch} and {\sc H. Sumi}, Multifractal Formalism for generalised local dimension spectra of Gibbs measures on the real line. {\em arXiv:1902.06962}, 2019.

\bibitem{Jaffard1}
{\sc S. Jaffard}, Multifractal formalism for functions part I: results valid for all functions. {\em SIAM J. Math. Anal.} {\bf 28} (1997), no. 4, 944--970.

\bibitem{Jaffard2}
{\sc S. Jaffard}, Multifractal formalism for functions part II: self-similar functions.  {\em SIAM J. Math. Anal.} {\bf 28} (1997), no. 4, 971--998.

\bibitem{JafMan}
{\sc S. Jaffard} and {\sc B. B. Mandelbrot}, Local regularity of nonsmooth wavelet expansions and application to the Polya function. {\em Adv. Math} {\bf 120} (1996), 265--282.

\bibitem{Katsuura} 
{\sc H. Katsuura}, Continuous nowhere-differentiable functions - an application of contraction mappings. {\em Amer. Math. Monthly} {\bf 98} (1991), no. 5, 411--416.

\bibitem{Kawamura}
{\sc K. Kawamura}, On the set of points where Lebesgue's singular function has the derivative zero. {\em Proc. Japan Acad. Ser. A Math. Sci.} {\bf 87} (2011), no. 9, 162--166.




\bibitem{Ledrappier} 
{\sc F. Ledrappier}, On the dimension of some graphs, {\em Contemp. Math.} {\bf 135} (1992), 285--293.



\bibitem{Li-Dekking}
{\sc W. Li} and {\sc F. M. Dekking}, Hausdorff dimension of subsets of Moran fractals with prescribed group frequency of their codings. {\em Nonlinearity} {\bf 16} (2003), 187--199.

\bibitem{Mandelbrot}
{\sc B. Mandelbrot}, Intermittent turbulence in self-similar cascades: Divergence of high moments
and dimension of the carrier. {\em J. Fluid Mech.} {\bf 62} (1974), 331.

\bibitem{Okamoto} 
{\sc H. Okamoto}, A remark on continuous, nowhere differentiable functions. {\em Proc. Japan Acad. Ser. A Math. Sci.} {\bf 81} (2005), no. 3, 47--50.




\bibitem{Riesz-Nagy}
{\sc F. Riesz} and {\sc B. Sz.-Nagy}, {\em Functional Analysis}, Ungar, New York, 1955.


\bibitem{Salem}
{\sc R. Salem}, On some singular monotonic functions which are strictly increasing. {\em Trans. Amer. Math. Soc.} {\bf 53} (1943), 427--439.


\bibitem{Takagi}
{\sc T.~Takagi}, A simple example of the continuous function without derivative, {\em Phys.-Math. Soc. Japan} {\bf 1} (1903), 176-177. {\em The Collected Papers of Teiji Takagi}, S. Kuroda, Ed., Iwanami (1973), 5--6. 

\end{thebibliography}
\end{document}